\documentclass{article}

\usepackage{microtype}
\usepackage{graphicx}
\usepackage{subfigure}
\usepackage{booktabs} 

\usepackage{hyperref}

\usepackage[accepted]{icml2024}

\usepackage{amsmath}
\usepackage{amssymb}
\usepackage{mathtools}
\usepackage{amsthm}

\usepackage[capitalize,noabbrev]{cleveref}

\theoremstyle{plain}
\usepackage{microtype}
\usepackage{graphicx}
\usepackage{subfigure}
\usepackage{booktabs} 
\usepackage{amsmath}
\usepackage{amssymb}
\usepackage{mathtools}
\usepackage{amsthm,fancybox}
\usepackage{paralist}
\usepackage{multirow}
\RequirePackage{algorithm}
\RequirePackage{algorithmic}

\newtheorem{theorem}{Theorem}
\newtheorem{cor}{Corollary}

\newtheorem{lemma}{Lemma}
\newtheorem{ass}{Assumption}
\newtheorem{definition}[cor]{Definition}
\newcommand{\Norm}[1]{\left\|#1\right\|}

\def \E {\mathbb{E}}
\def \R {\mathbb{R}}

\def \r {\mathbf{r}}

\def \u {\mathbf{u}}
\def \v {\mathbf{v}}
\def \s {\mathbf{s}}
\def \w {\mathbf{w}}

\def \F {\mathcal{F}}
\def \G {\mathcal{G}}
\def \z {\mathbf{z}}
\def \x {\mathbf{x}}
\def \y {\mathbf{y}}

\def \LL {\mathcal{L}}

\def \X {\mathcal{X}}
\usepackage[textsize=tiny]{todonotes}

\icmltitlerunning{Projection-Free Variance Reduction Methods for Stochastic Constrained Multi-Level Compositional Optimization}

\begin{document}

\twocolumn[
\icmltitle{Projection-Free Variance Reduction Methods for Stochastic Constrained Multi-Level Compositional Optimization}




\begin{icmlauthorlist}
\icmlauthor{Wei Jiang}{nju}
\icmlauthor{Sifan Yang}{nju,njuai}
\icmlauthor{Wenhao Yang}{nju,njuai}
\icmlauthor{Yibo Wang}{nju,njuai}
\icmlauthor{Yuanyu Wan}{zju,nju}
\icmlauthor{Lijun Zhang}{nju,njuai}
\end{icmlauthorlist}

\icmlaffiliation{nju}{National Key Laboratory for Novel Software Technology, Nanjing University, Nanjing, China}
\icmlaffiliation{njuai}{School of Artificial Intelligence, Nanjing University, Nanjing, China}
\icmlaffiliation{zju}{School of Software Technology, Zhejiang University, Ningbo, China}

\icmlcorrespondingauthor{Lijun Zhang}{zhanglj@lamda.nju.edu.cn}

\icmlkeywords{Machine Learning, ICML}

\vskip 0.3in
]



\printAffiliationsAndNotice{}  

\begin{abstract}
This paper investigates projection-free algorithms for stochastic constrained multi-level  optimization. In this context, the objective function is a nested composition of several smooth functions, and the decision set is closed and convex. Existing projection-free algorithms for solving this problem suffer from two limitations: 1) they solely focus on the gradient mapping criterion and fail to match the optimal sample complexities in unconstrained settings; 2) their analysis is exclusively applicable to non-convex functions, without considering convex and strongly convex objectives. To address these issues, we introduce novel projection-free variance reduction algorithms and analyze their complexities under different criteria. For gradient mapping, our complexities improve existing results and match the optimal rates for unconstrained problems. For the widely-used Frank-Wolfe gap criterion, we provide theoretical guarantees that align with those for single-level problems. Additionally, by using a stage-wise adaptation, we further obtain complexities for convex and strongly convex functions. Finally, numerical experiments on different tasks demonstrate the effectiveness of our methods.
\end{abstract}

\section{Introduction}
In this paper, we consider projection-free algorithms for  stochastic constrained multi-level compositional optimization in the form of
\begin{equation} \label{prob:1}
\min_{\x\in \X} F(\x) = f_K \circ f_{K-1} \circ \cdots \circ f_1(\x),
\end{equation}
where $\X$ is a closed convex set. We assume that each function $f_i$ and its gradient can only be accessed through noisy estimations, symbolized as $f_i(\cdot;\xi)$ and $\nabla f_i(\cdot;\xi)$ such that
\begin{align*}
\E_\xi \left[ f_i(\cdot;\xi) \right] = f_i(\cdot) \quad \text{and} \quad \E_\xi \left[\nabla f_i(\cdot;\xi) \right] = \nabla f_i(\cdot),
\end{align*}
where $\xi$ denotes samples drawn from the oracle. Problem~(\ref{prob:1}) finds wide applications in machine learning tasks, such as reinforcement learning~\citep{Dann2014PolicyEW}, government planning~\citep{Bruno2016RiskNA}, risk management~\citep{ Dentcheva2015StatisticalEO}, model-agnostic meta-learning~\citep{Ji2020MultiStepMM}, robust learning~\citep{li2021tilted}, risk-averse portfolio optimization~\citep{Shapiro2009LecturesOS}, and graph neural network training~\citep{balasubramanian2020stochastic}.

Although stochastic multi-level optimization has been investigated extensively in recent years~\citep{Yang2019MultilevelSG,balasubramanian2020stochastic,Zhang2021MultiLevelCS,chen2021solving,jiang2022optimal}, current work mainly focuses on unconstrained problems, i.e., $\X = \R^d$. For many practical problems, such as risk-averse portfolio optimization, the decision set is constrained~(e.g., the decision variable $\x$ should be in a simplex for portfolio optimization). Traditional constrained optimization typically employs a projection operation to ensure that the solutions are within the decision set. However, projection is usually complicated and time-consuming, and existing literature~\citep{NEURIPS2022_7e16384b} begin to show interest in developing projection-free algorithms for constrained multi-level problems by replacing projection~(a convex optimization problem) with multiple steps of more efficient linear minimization operation.

\begin{table*}[t]
\caption{Summary of results for projection-free algorithms under three different criteria: the Frank-Wolfe gap~(FG), gradient mapping~(GM), and optimal gap~(OG). Here CVX represents convex functions and SC stands for $\lambda$-strongly convex functions. We compare our methods with 1-SFW~\citep{Zhang2019OneSS},SPIDER-FW~\citep{pmlr-v97-yurtsever19b}, NCGS~\citep{pmlr-v80-qu18a}, SGD+ICG~\citep{Balasubramanian2018ZerothOrderNS}, LiNASA+ICG~\citep{NEURIPS2022_7e16384b}, and SCGS~\citep{doi:10.1137/140992382}.}
\label{t1}
\vskip 0.15in
\begin{center}
\begin{small}
\begin{tabular}{lcccccr}
\toprule
Method & Criterion & Assumptions & Level & Batch size & SFO & LMO \\
\midrule
1-SFW    & FG &Smooth & $1$ &  $1$ &$\mathcal{O}\left(\epsilon^{-3}\right)$ & $\mathcal{O}\left(\epsilon^{-3}\right)$\\
SPIDER-FW   & FG& Smooth & $1$ &  $\mathcal{O}\left(\epsilon^{-1}\right)$ &$\mathcal{O}\left(\epsilon^{-3}\right)$ & $\mathcal{O}\left(\epsilon^{-2}\right)$ \\
\textbf{Theorem~\ref{thm:main}} & FG &Smooth&   $K$ & $1$ &$\mathcal{O}\left(\epsilon^{-3}\right)$ & $\mathcal{O}\left(\epsilon^{-3}\right)$\\
\textbf{Theorem~\ref{thm:main_0}} & FG &Smooth&   $K$ & $\mathcal{O}\left(\epsilon^{-1}\right)$ &$\mathcal{O}\left(\epsilon^{-3}\right)$ & $\mathcal{O}\left(\epsilon^{-2}\right)$\\
\midrule
NCGS & GM &Smooth &  $1$ &$\mathcal{O}\left(\epsilon^{-1}\right)$ &$\mathcal{O}\left(\epsilon^{-2}\right)$ & $\mathcal{O}\left(\epsilon^{-2}\right)$\\
SGD+ICG  & GM & Smooth &  $1$ & $\mathcal{O}\left(\epsilon^{-1}\right)$ &$\mathcal{O}\left(\epsilon^{-2}\right)$ & $\mathcal{O}\left(\epsilon^{-2}\right)$\\
LiNASA+ICG   & GM &Smooth &  $K$ & 1 &$\mathcal{O}\left(\epsilon^{-2}\right)$ & $\mathcal{O}\left(\epsilon^{-3}\right)$ \\
\textbf{Theorem~\ref{thm:2}}  & GM & Smooth & $K$ & $1$ &$\mathcal{O}\left(\epsilon^{-1.5}\right)$ & $\mathcal{O}\left(\epsilon^{-2.5}\right)$\\
\textbf{Theorem~\ref{thm:2_0}}  & GM & Smooth & $K$ & $\mathcal{O}\left(\epsilon^{-0.5}\right)$ &$\mathcal{O}\left(\epsilon^{-1.5}\right)$ & $\mathcal{O}\left(\epsilon^{-2}\right)$\\
\midrule
1-SFW & OG &Smooth+CVX &   $1$ & $1$ &$\mathcal{O}\left(\epsilon^{-2}\right)$ & $\mathcal{O}\left(\epsilon^{-2}\right)$ \\
SPIDER-FW & OG &Smooth+CVX &   $1$ & $\mathcal{O}\left(\epsilon^{-1}\right)$ &$\mathcal{O}\left(\epsilon^{-2}\right)$ & $\mathcal{O}\left(\epsilon^{-1}\right)$ \\
\textbf{Theorem~\ref{thm:3}} & OG &Smooth+CVX &   $K$ & $1$ &$\mathcal{O}\left(\epsilon^{-2}\right)$ & $\mathcal{O}\left(\epsilon^{-2}\right)$ \\
\textbf{Theorem~\ref{thm:3_0}} & OG &Smooth+CVX &   $K$ & $\mathcal{O}\left(\epsilon^{-1}\right)$ &$\mathcal{O}\left(\epsilon^{-2}\right)$ & $\mathcal{O}\left(\epsilon^{-1}\right)$ \\
\midrule
SCGS
& OG &Smooth+SC &   $1$ & $\mathcal{O}\left(\epsilon^{-1}\right)$ &$\mathcal{O}\left(\lambda^{-1}\epsilon^{-1}\right)$ & $\mathcal{O}\left(\epsilon^{-1}\right)$ \\
\textbf{Theorem~\ref{thm:4}} & OG &Smooth+SC &   $K$ & $1$ &$\mathcal{O}\left(\lambda^{-1}\epsilon^{-1}\right)$ & $\mathcal{O}\left(\epsilon^{-2}\right)$ \\
\textbf{Theorem~\ref{thm:4_0}}
& OG &Smooth+SC &   $K$ & $\mathcal{O}\left(\epsilon^{-1}\right)$ &$\mathcal{O}\left(\lambda^{-1}\epsilon^{-1}\right)$ & $\mathcal{O}\left(\epsilon^{-1}\right)$ \\
\bottomrule
\\
\end{tabular}
\end{small}
\end{center}
\vskip -0.1in
\end{table*}
Projection-free methods typically require two oracles: 1) the Stochastic First-order Oracle~(SFO), which takes a point $\x$ and returns the pair $(f(\x;\xi),\nabla f(\x;\xi))$, where $\xi$ is a sample drawn from the oracle; 2) the Linear Minimization Oracle~(LMO), which takes a direction $\mathbf{d}$ and outputs $\arg \min_{\x \in \X} \left\langle \x , \mathbf{d} \right\rangle $. To evaluate different projection-free algorithms, the most widely used measures are the number of calls to SFO and LMO required to attain an acceptable solution. For non-convex functions, such a solution $\x$ is usually defined by the Frank-Wolfe gap~\cite{lacostejulien2016convergence}, formalized as 
\begin{align}\label{FW}
    \F(\x) \coloneqq \max_{\hat{\x} \in \X}\langle\hat{\x}-\x,-\nabla F(\x)\rangle \leq \epsilon,
\end{align}
where $\epsilon$ is a small value. More recently, the gradient mapping criterion~\citep{pmlr-v80-qu18a} has been introduced, which is denoted as \begin{align}\label{GM}
    \G(\x) \coloneqq \Norm{\beta\left(\x-\Pi_{\mathcal{X}}\left(\x-\frac{1}{\beta}\nabla F(\x)\right)\right)}^2\leq \epsilon,
\end{align}
where $\Pi_{\X}$ denote the projection onto the domain $\X$ and parameter $\beta$ is a positive constant. Notably, if $\X=\R^d$, this gradient mapping criterion simplifies to the stationary point, i.e., $\E\left[\Norm{\nabla F(\x)}^2\right]\leq \epsilon$, which is the standard metric for stochastic unconstrained problems. When the objective function is convex or strongly convex, the optimal gap criterion is employed instead, expressed as 
\begin{align}\label{OP}
    F(\x)-\min_{\hat{\x} \in \X} F(\hat{\x})\leq \epsilon,
\end{align}
which measures the difference between the objective and the optimal value.

The current method for stochastic projection-free multi-level compositional optimization, named as LiNASA+ICG~\citep{NEURIPS2022_7e16384b}, integrates the linearized NASA~\citep{Ghadimi2020AST} algorithm with inexact conditional gradient~\citep{Balasubramanian2018ZerothOrderNS}, and provides theoretical guarantees under the gradient mapping criterion. This method can identify an acceptable point with $\mathcal{O}(\epsilon^{-2})$ calls to SFO and $\mathcal{O}(\epsilon^{-3})$ calls to LMO. However, it still suffers from several drawbacks. Firstly, its SFO complexity does not match the existing optimal rate of $\mathcal{O}(\epsilon^{-1.5})$ for stochastic unconstrained problems. Secondly, its analysis solely concentrates on gradient mapping, and results for the more widely used Frank-Wolfe gap criterion are not provided. Finally, this method is restricted to non-convex objective functions, and it is unclear how to improve the rate for convex and strongly convex objectives.

To address these issues, we propose new algorithms that utilize the variance reduction estimator STORM~\citep{cutkosky2019momentum} to obtain more accurate evaluations of both the inner function values and the overall gradient. By integrating these estimators with a specifically designed Frank-Wolfe algorithm~\citep{pmlr-v28-jaggi13}, we improve the rate for gradient mapping and are able to analyze the Frank-Wolfe gap. Besides, by employing a large batch size, we can reduce the iteration numbers and thus improve the LMO complexities while maintaining the same SFO rates. Moreover, we develop a stage-wise algorithm with a warm-start technique and provide theoretical guarantees for both convex and strongly convex functions. Compared with previous methods, this paper makes the following contributions:
\begin{compactenum}
\item We establish the first theoretical guarantees for the Frank-Wolfe gap~(Theorem~\ref{thm:main},~\ref{thm:main_0}) under the multi-level setting. The rates we obtained match the results for single-level projection-free methods~\citep{Zhang2019OneSS,pmlr-v97-yurtsever19b}.

\vspace*{0.05in}\item For gradient mapping~(Theorem~\ref{thm:2},~\ref{thm:2_0}), our approach achieves an improved SFO complexity of $\mathcal{O}(\epsilon^{-1.5})$ and LMO complexity of $\mathcal{O}(\epsilon^{-2.5})$. The SFO complexity matches the low bound~\citep{Arjevani2019LowerBF}, and the LMO complexity can be further reduced to $\mathcal{O}(\epsilon^{-2})$ by using large batch sizes.

\vspace*{0.05in}\item We explore the complexities for convex~(Theorem~\ref{thm:3},~\ref{thm:3_0}) and strongly convex functions~(Theorem~\ref{thm:4},~\ref{thm:4_0}), and derive the optimal SFO rates for these problems, which have not been studied in previous projection-free multi-level literature.
\end{compactenum}

We compare our theoretical results with existing methods in Table~\ref{t1}, and validate the effectiveness of our method through numerical experiments in Section~\ref{sec:4}.

\section{Related Work}
This section briefly reviews related work on stochastic multi-level compositional optimization and stochastic projection-free algorithms.

\subsection{Stochastic Multi-Level Compositional Optimization}
Stochastic Compositional Optimization has been explored extensively in the literature, and most research focuses on the two-level settings~\cite{wang2017stochastic,DBLP:journals/jmlr/WangLF17,Ghadimi2020AST,Zhang2019ASC,chen2021solving,qi2021online,jiang2022multiblocksingleprobe,ICML:2023:Jiang,yu2024efficient}. The problem of multi-level compositional optimization was first investigated by \citet{Yang2019MultilevelSG}. Inspired by multi-timescale stochastic approximation~\cite{wang2017stochastic}, they introduced the multi-level stochastic gradient method, which achieves a sample complexity of $\mathcal{O}\left(1 / \epsilon^{(7+K)/2}\right)$ for $K$-level problems. When the function is strongly convex, this complexity can be further improved to $\mathcal{O}\left(1 / \epsilon^{(3+K)/4}\right)$. Subsequently, motivated by the NASA algorithm~\citep{Ghadimi2020AST}, \citet{balasubramanian2020stochastic} proposed using a linearized averaging stochastic estimator to track the function value, attaining a sample complexity of $\mathcal{O}\left(1 / \epsilon^{4}\right)$ for non-convex objectives. This rate was also obtained in a concurrent work~\citep{chen2021solving} by employing variance reduction techniques to evaluate the function value.

Later, \citet{Zhang2021MultiLevelCS} employed nested variance reduction to approximate gradients, improving the sample complexity to the optimal $\mathcal{O}\left(1 / \epsilon^{3}\right)$. However, this approach requires a large and increasing batch size on the order of $\mathcal{O}\left(1 / \epsilon\right)$. To address this issue, \citet{jiang2022optimal} developed a method called SMVR, which achieves the same optimal rate but only requires a constant batch size. SMVR also attains an improved rate of $\mathcal{O}\left(1 / \epsilon^{2}\right)$ for convex functions and $\mathcal{O}\left(1 / \left(\lambda \epsilon\right)\right)$ for $\lambda$-strongly convex objectives. More recently, \citet{gao2023stochastic} further introduced the decentralized stochastic multi-level optimization algorithm, which achieves the level-independent convergence rate under the decentralized setting. Despite these advancements, these algorithms are only applicable to unconstrained problems.

\subsection{Stochastic Projection-Free Algorithms}
The most well-known projection-free method,  Frank-Wolfe algorithm~\cite{Frank1956AnAF}, was originally designed for smooth convex optimization with polyhedral domains and has been extended to any convex compact set by~\citet{pmlr-v28-jaggi13}. 
In the stochastic setting, \citet{hazan2012} first developed a projection-free method for online smooth convex optimization. Later, \citet{pmlr-v48-hazana16} applied variance reduction techniques to the stochastic Frank-Wolfe algorithm. Inspired by the accelerated gradient method~\citep{Nesterov1983AMF}, \citet{doi:10.1137/140992382} proposed the stochastic conditional gradient sliding~(SCGS) method, which offered an SFO complexity of $\mathcal{O}\left(\lambda^{-1}\epsilon^{-1}\right)$ and an LMO complexity of $\mathcal{O}\left(\epsilon^{-1}\right)$ for smooth $\lambda$-strongly convex optimization. Besides that, projection-free methods are also widely investigated in online convex optimization in recent years~\cite{Hazan20,ICML:2020:Wan,JMLR:2022:Wan,Zak_SC22,Garber23,AAAI:2024:Wang}.

For non-convex objectives, \citet{Reddi2016StochasticFM} introduced the SVFW method, achieving an SFO complexity of $\mathcal{O}(\epsilon^{-10/3})$ and an LMO complexity of $\mathcal{O}(\epsilon^{-2})$ under the Frank-Wolfe gap. Motivated by the variance reduction technique SPIDER~\citep{Fang2018SPIDERNN}, \citet{pmlr-v97-yurtsever19b} developed the SPIDER-FW algorithm, improving the SFO complexity to $\mathcal{O}(\epsilon^{-3})$ by using a large batch size of $\mathcal{O}(\epsilon^{-1})$. To avoid relying on large batches, \citet{Zhang2019OneSS} proposed the one-sample stochastic Frank-Wolfe algorithm~(1-SFW), which attains the same SFO complexity and obtains an LMO complexity of $\mathcal{O}(\epsilon^{-3})$. Rather than focusing on the Frank-Wolfe gap criterion, \citet{pmlr-v80-qu18a} and \citet{Balasubramanian2018ZerothOrderNS} explored the gradient mapping criterion, and attained $\mathcal{O}(\epsilon^{-2})$ complexities for both the SFO and LMO at the cost of using a large batch of $\mathcal{O}(\epsilon^{-1})$ in each iteration. 

In the context of stochastic multi-level optimization, \citet{NEURIPS2022_7e16384b} recently proposed a projection-free conditional gradient-type algorithm, which combines the linearized NASA algorithm with the inexact conditional gradient technique~\citep{Balasubramanian2018ZerothOrderNS}. This method achieves an SFO complexity of $\mathcal{O}\left(1 / \epsilon^{2}\right)$ and an LMO complexity of $\mathcal{O}\left(1 / \epsilon^{3}\right)$ in terms of the gradient mapping criterion. However, its SFO complexity does not match the complexity of $\mathcal{O}\left(1/\epsilon^{1.5}\right)$ achieved by variance reduction methods for unconstrained multi-level problems\footnote{Gradient mapping reduces to $\E\left[\Norm{\nabla F(\x)}^2 \right] \leq \epsilon$ for unconstrained problems, and existing methods~\cite{Zhang2021MultiLevelCS,jiang2022optimal} ensure $\E\left[\Norm{\nabla F(\x)} \right] \leq \epsilon$ with a complexity of $\mathcal{O}(1/\epsilon^{3})$, implying a rate of $\mathcal{O}(1/\epsilon^{1.5})$ for gradient mapping.}. Furthermore, they only consider non-convex functions, and solely focus on the gradient mapping criterion, which are the main drawbacks we aim to address in this paper.

\section{The Proposed Methods}
In this section, we first introduce the assumptions used in this paper. Then, we present the proposed algorithms, along with their corresponding theoretical guarantees for three different criteria: Frank-Wolfe gap, gradient mapping and optimal gap.

\subsection{Assumptions}
We adopt the following assumptions throughout the paper, which are commonly used in studies of stochastic compositional optimization~\citep{Yuan2019EfficientSN,Zhang2019ASC,Zhang2021MultiLevelCS,jiang2022optimal} and stochastic projection-free analysis~\citep{pmlr-v80-qu18a, pmlr-v97-yurtsever19b, Zhang2019OneSS}.

\begin{ass} \label{asm:0} (Constrained set)
The decision set $\X$ is closed and convex with a bounded domain such that $\max _{\x, \y \in \mathcal{X}}\Norm{\x-\y} \leq D$.
\end{ass}

\begin{ass} \label{asm:1} (Smoothness and Lipschitz continuity)
All functions $f_1, \dotsc, f_K$ are $L_f$-Lipschitz continuous, and their Jacobians $\nabla f_1, \dotsc, \nabla f_K$ are $L_J$-Lipschitz continuous.
\end{ass}

\begin{ass}\label{asm:stochastic2} (Bounded variance)\ \ For $1\leq i\leq K$, we assume that:
\begin{gather*}
       \E_{\xi_t^i}\left[f_i(\x;\xi_t^i)\right] = f_i(\x), 
	 \\ \E_{\xi_t^i}\left[\nabla f_i(\x;\xi_t^i)\right] = \nabla f_i(\x),\\
 \E_{\xi_t^i}\left[\Norm{f_i(\x;\xi_t^i) - f_i(\x)}^2 \right]\leq \sigma^2,
	 \\
	\E_{\xi_t^i}\left[\Norm{\nabla f_i(\x;\xi_t^i) - \nabla f_i(\x) }^2 \right] \leq  \sigma_J^2,
\end{gather*}
where $\{\xi_t^i\}_{i=1}^K$ are mutually independent.
\end{ass}

\begin{ass}\label{asm:stoc_smooth3}
	(Average smoothness) \ \ For $1\leq i\leq K$, we assume that:
	\begin{align*}
	\E_{\xi_t^i}\left[\Norm{f_i(\x;\xi_t^i)-  f_i(\y;\xi_t^i)}^2\right] &\leq \LL_f^2 \Norm{\x - \y}^2,\\
		\E_{\xi_t^i}\left[\Norm{\nabla f_i(\x;\xi_t^i)- \nabla f_i(\y;\xi_t^i)}^2\right] &\leq  \LL_J^2 \Norm{\x - \y}^2.
	\end{align*}
\end{ass}

\begin{ass}\label{asm:stochastic4} We suppose that $F\left(\x_{1}\right)-F_{\star} \leq \Delta_{F}$ for the initial solution $\x_{1}$, where $F_{\star}=\min_{\x \in \X} F(\x)$.
\end{ass}

\subsection{Results for Frank-Wolfe Gap}
First, we examine the sample complexity under the criterion of Frank-Wolfe gap. The primary procedure of our algorithm involves estimating the gradient of the objective function and then employing the Frank-Wolfe method to replace the projection operation. Note that the gradient of the multi-level function exhibits a nested structure, and the estimation error would accumulate as the level becomes deeper. To address this issue, we resort to the variance reduction technique STORM~\citep{cutkosky2019momentum} to estimate both the inner function value and the overall gradient. Specifically, we draw a batch of samples $\{\xi_t^{i,1},\cdots ,\xi_t^{i,B_1} \}$ with the batch size $B_1$ for each level $i$ at time step $t$, and then employ a variance reduction estimator $\u_t^i$ to track each inner function value $f_i(\cdot)$ as:
\begin{equation}
    \begin{aligned}\label{ss1}
    \u_t^i = (1-\alpha)\u_{t-1}^i +   \frac{1}{B_1} \sum_{j=1}^{B_1} f_i(\u_t^{i-1};\xi_t^{i,j})\\
    -  (1-\alpha)\frac{1}{B_1} \sum_{j=1}^{B_1}f_i(\u_{t-1}^{i-1};\xi_t^{i,j}) .
\end{aligned}
\end{equation}
This evaluation ensures that the estimation error is reduced over time. Then, we employ a similar variance-reduced estimator $\v_t$ to evaluate the overall gradient of the objective function $\nabla F(\x_t)$ as:
\begin{equation}
\begin{aligned}\label{ss2}
\v_t = (1-\alpha)\v_{t-1} +\frac{1}{B_1} \sum_{j=1}^{B_1} \left[ \prod_{i=1}^K \nabla f_i(\u_t^{i-1};\xi_t^{i,j}) \right] \\
- (1-\alpha)\frac{1}{B_1} \sum_{j=1}^{B_1} \left[\prod_{i=1}^K \nabla f_i(\u_{t-1}^{i-1};\xi_t^{i,j})\right] .
\end{aligned}
\end{equation}
After obtaining the gradient estimation, we follow the framework of the Frank-Wolfe algorithm~\citep{pmlr-v28-jaggi13}, but use the estimator $\v_t$ to replace the gradient required in the original algorithm as follows:
\begin{gather*}
    \z_{t} = \arg \min_{\x \in \X} \langle \x,\v_{t} \rangle, \\  \x_{t+1} = \x_t + \eta (\z_t - \x_t).
\end{gather*}
In this way, we develop our Projection-free Multi-level Variance Reduction~(PMVR) method for stochastic multi-level problems. The complete algorithm is presented in Algorithm~\ref{alg:1}. Note that in the first iteration~(when $t=1$), we can simply set estimator $\u_{1}^i = \frac{1}{B_0} \sum_{j=1}^{B_0}  f_i(\u_{1}^{i-1};\xi^{i,j}_1)$ and $\v_{1} = \frac{1}{B_0} \sum_{j=1}^{B_0} \left[ \prod_{i=1}^K  \nabla f_i(\u_{1}^{i-1};\xi_1^{i,j}) \right]$, where $B_0$ is the batch size for the first iteration.

\begin{algorithm}[!tb]
	\caption{PMVR Algorithm}
	\label{alg:1}
	\begin{algorithmic}[1]
	\STATE {\bfseries Input:} parameters $T$, $\eta$, $\alpha$, initial points $\left(\x_1,\u_1,\v_1\right)$ 
		\FOR{time step $t = 1$ {\bfseries to} $T$}
		\STATE Set $\u_t^0 = \x_t$
		\FOR{level $i = 1$ {\bfseries to} $K$}
		\STATE Draw a batch of samples $\left\{ \xi_t^{i,1}, \dots, \xi_t^{i,B_1} \right\}$
		\STATE Compute the estimator $\u_t^i$ according to~(\ref{ss1})
        \ENDFOR
        \STATE Compute the gradient estimator $\v_{t}$ according to~(\ref{ss2})
        \STATE Compute $\z_{t} = \arg \min_{\x \in \X} \langle \x,\v_{t} \rangle$
		\STATE Update the weight: $\x_{t+1} = \x_t + \eta (\z_t - \x_t) $
		\ENDFOR
	\STATE Choose $\tau$ uniformly at random from the set $\{1, \ldots, T\}$
	\STATE Return $(\x_{\tau},\u_{\tau},\v_{\tau})$
	\end{algorithmic}
\end{algorithm}

\textbf{Comparison with the SMVR method:} Although the SMVR algorithm~\citep{jiang2022optimal} for unconstrained problems also utilizes the STORM estimator to assess inner function values and gradients at each level, our PMVR method differs from  SMVR in the following aspects: 
i) SMVR first estimates the gradient at each level using STORM and then computes the overall gradient through multiplication, requiring an additional gradient clipping for each level to ensure the overall gradient does not explode. This operation demands knowledge of the gradient upper bound of each level, which is impractical in real-world applications. In contrast, our PMVR method directly applies variance reduction to the overall gradient, eliminating the need for such gradient clipping. This technique is also mentioned in a concurrent work~\cite{gao2023stochastic}; 
ii) Our method employs a constant learning rate $\eta$ and momentum parameter $\alpha$, which are easy to implement and fine-tune, whereas SMVR requires these parameters to decrease gradually, which is more complicated; 
iii) Our PMVR is a projection-free algorithm specifically designed for constrained problems, with theoretical guarantees focused on the Frank-Wolfe gap and gradient mapping. In contrast, SMVR aims to find stationary points for unconstrained objectives.

Next, we present the sample complexities of Algorithm~\ref{alg:1} with respect to the Frank-Wolfe gap $\F(\cdot)$ defined in equation~(\ref{FW}). Note that by using a large batch size~$B_1$, we are able to decrease the iteration numbers and thus reduce the LMO complexity. As a result, we provide a constant batch version and a large batch version for our guarantees.

\begin{theorem}\label{thm:main}
By setting $B_1 = \mathcal{O}(1)$, $\eta = \mathcal{O}(\epsilon^2)$, and $\alpha = \mathcal{O}(\epsilon^2)$, our PMVR algorithm guarantees that $\E\left[\F(\x_\tau)\right] \leq \epsilon$ within $T = \mathcal{O}(\epsilon^{-3})$ iterations.
\end{theorem}
\textbf{Remark:} The above theorem indicates that both the SFO and LMO complexities are of the order $\mathcal{O}(\epsilon^{-3})$, consistent with results for projection-free single-level problems using a constant batch size~\citep{Zhang2019OneSS}.
\begin{theorem}\label{thm:main_0} (Large Batch) By setting $B_1 = \mathcal{O}(\epsilon^{-1})$, $\eta = \mathcal{O}(\epsilon)$, and $\alpha = \mathcal{O}(\epsilon)$, our PMVR algorithm guarantees that $\E\left[\F(\x_\tau)\right] \leq \epsilon$ within $T = \mathcal{O}(\epsilon^{-2})$ iterations.
\end{theorem}
\textbf{Remark:} By employing a large batch size of $\mathcal{O}(\epsilon^{-1})$, our method obtains an SFO complexity of $\mathcal{O}(\epsilon^{-3})$ and an LMO complexity of $\mathcal{O}(\epsilon^{-2})$. These results align with those achieved by existing projection-free methods for single-level objectives~\citep{pmlr-v97-yurtsever19b}.

\subsection{Results for Gradient Mapping}
Then, we investigate the complexities under the criterion of gradient mapping $\G(\cdot)$ defined in equation~(\ref{GM}). To deal with the gradient mapping, our PMVR algorithm only requires minimal modifications to fit this criterion. Based on Proposition 2 of~\citet{NEURIPS2022_7e16384b}, gradient mapping can be decomposed into two components:
\begin{align*}
  \mathcal{G}(\x_t) \leq - 4 \beta \min _{\w \in \X} g(\w, \x_t, \v_t)+2\left\| \nabla F(\x_t)- \v_t \right\|^{2},  
\end{align*}  
where  
\begin{align*}
  g(\w, \x_t, \v_t)= \langle \v_t, \w-\x_t \rangle+\frac{\beta}{2}\|\w-\x_t\|^{2}. 
\end{align*}
Since our PMVR method has already employed variance-reduced techniques to ensure that the gradient estimation error $\E\left[\Norm{\nabla F(\x_t)- \v_t }^{2}\right]$ decreases over time, we can reuse this part and focus on bounding the $-\min _{\w \in \X} g(\w, \x_t, \v_t)$ term. To address this constrained quadratic minimization problem, we design a sub-algorithm, modified from the Frank-Wolfe method~\citep{pmlr-v28-jaggi13}, which runs for several loops. In each loop $n$, we update as follows:
\begin{gather*}
    \s
      = \arg \min _{\hat{\s} \in \X}\left\langle \v_t+\beta\left(\w_n-\x_t\right), \hat{\s}\right\rangle, \\
      \w_{n+1}=\left(1-\gamma_{t}\right) \w_{n}+\gamma_{t} \s.
\end{gather*}
Note that $\v_t+\beta\left(\w_n-\x_t\right)$ is the gradient of $g(\w_n,\x_t,\v_t)$ with respect to parameter $\w_n$. Overall, we only need to replace Step 9 in Algorithm~\ref{alg:1} with a sub-algorithm presented in Algorithm~\ref{alg:2}, and we can obtain the guarantees for gradient mapping below.
\begin{algorithm}[!tb]
	\caption{PMVR-v2}
	\label{alg:2}
         Replace Step 9 of Algorithm~\ref{alg:1} with the following
 
	\begin{algorithmic}[1]
         \STATE Initialize $\w_1 = \x_t$
        \FOR{time step $n = 1$ {\bfseries to} $N$}
	\STATE Compute $ \s
      = \arg \min _{\hat{\s} \in \X}\left\langle \v_t+\beta\left(\w_n-\x_t\right), \hat{\s}\right\rangle$
	\STATE Set $\w_{n+1}=\left(1-\gamma_{t}\right) \w_{n}+\gamma_{t} \s$
	\ENDFOR
     \STATE Set $\z_t = \w_{N+1}$
\end{algorithmic}
\end{algorithm}

\begin{theorem}\label{thm:2}
By setting $B_1 = \mathcal{O}(1)$, $N = \mathcal{O}(\epsilon^{-1})$, $\eta = \mathcal{O}(\sqrt{\epsilon})$, and $\alpha = \mathcal{O}(\epsilon)$, our PMVR-v2 algorithm guarantees $ \E\left[\G (\x_\tau) \right]\leq \epsilon$ in $T=\mathcal{O}(\epsilon^{-1.5})$ iterations.
\end{theorem}
\textbf{Remark:} When using a constant batch size, our algorithm results in $\mathcal{O}(\epsilon^{-1.5})$ SFO complexity and $\mathcal{O}(\epsilon^{-2.5})$ LMO complexity, which are both better than the previous algorithm LiNASA+ICG~\citep{NEURIPS2022_7e16384b}, that achieves $\mathcal{O}(\epsilon^{-2})$  SFO complexity and $\mathcal{O}(\epsilon^{-3})$  LMO complexity. 

Next, we can improve the LMO complexity by using a large batch size.
\begin{theorem}\label{thm:2_0} (Large Batch) By setting $B_1 = \mathcal{O}(\epsilon^{-0.5})$, $N = \mathcal{O}(\epsilon^{-1})$, $\eta = \mathcal{O}(1)$, and $\alpha = \mathcal{O}(\sqrt{\epsilon})$, our PMVR-v2 ensures $\E\left[\G (\x_\tau) \right] \leq \epsilon$ in $T=\mathcal{O}(\epsilon^{-1})$ iterations.
\end{theorem}
\textbf{Remark:} The above theorem indicates that PMVR-v2, with a batch size of $\mathcal{O}(\epsilon^{-0.5})$, achieves an SFO complexity of $\mathcal{O}(\epsilon^{-1.5})$ and an LMO complexity of $\mathcal{O}(\epsilon^{-2})$. These rates are superior to methods for single-level objectives, i.e., NCGS~\citep{pmlr-v80-qu18a} and SGD+ICG~\citep{Balasubramanian2018ZerothOrderNS}, which require a larger batch size of $\mathcal{O}(\epsilon^{-1})$ and demand a worse SFO complexity of $\mathcal{O}\left(\epsilon^{-2}\right)$. Notably, our SFO complexity of $\mathcal{O}(\epsilon^{-1.5})$ also matches the lower bound for stochastic unconstrained optimization~\citep{Arjevani2019LowerBF}.

\subsection{Results for Optimal Gap}
In addition to the analyses for general non-convex functions in previous subsections, we further investigate the case for convex and strongly convex functions, defined as follows.
\begin{definition}
A function $F:\X \mapsto \mathbb{R} $ is convex if
\begin{align*}
  F(\mathbf{y}) \geq F(\mathbf{x})+\langle\nabla F(\mathbf{x}), \mathbf{y}-\mathbf{x}\rangle, \forall \mathbf{x}, \mathbf{y} \in \X.
\end{align*}
\end{definition}

\begin{definition}\label{def:sc}
A function $F:\X \mapsto \mathbb{R} $ is $\lambda$-strongly convex if 
 $\ \forall \mathbf{x}, \mathbf{y} \in \X$
\begin{align*}
  F(\mathbf{y}) \geq F(\mathbf{x})+\langle\nabla F(\mathbf{x}), \mathbf{y}-\mathbf{x}\rangle+\frac{\lambda}{2}\|\mathbf{y}-\mathbf{x}\|^{2}.
\end{align*}
\end{definition}
When the objective function is convex or strongly convex, a local optimal point becomes a global optimal point. Consequently, a natural criterion to consider is the optimal gap defined in equation~(\ref{OP}).

\paragraph{Convex functions:}First, we investigate the case for convex functions. To obtain theoretical guarantees, we design a stage-wise algorithm with the warm-start technique building on Algorithm~\ref{alg:1}. Specifically, we divide the entire process into $S$ stages, and for each stage $s$, we run Algorithm~\ref{alg:1} with a new set of parameters $T_s, \eta_s, \alpha_s$, and use the output of the previous stage $\left\{\x_{s-1}, \u_{s-1}, \v_{s-1}\right\}$ as the initial points. The complete method is shown in Algorithm~\ref{alg:3}, referred to as Stage-wise PMVR.
\begin{algorithm}[!tb]
	\caption{Stage-wise PMVR}
	\label{alg:3}
	\begin{algorithmic}[1]
	\STATE {\bfseries Input:} initial points $\left(\x_0,\u_0,\v_0\right)$
		\FOR{stage $s = 1$ {\bfseries to} $S$}
		\STATE $(\x_{s},\u_{s},\v_{s})$ = Algorithm~1 with parameters $T_{s}$, $\eta_s$, $\alpha_s$ and initial points $\left(\x_{s-1},\u_{s-1},\v_{s-1}\right)$
		\ENDFOR
	\STATE Return $\x_{S}$
	\end{algorithmic}
\end{algorithm}

By decreasing $\alpha_s$ and $\eta_s$, and increasing $T_s$ at each stage, we can ensure that the optimal gap is halved after each stage, such that $\E\left[F(\x_s) - F_{\star}\right] \leq \frac{1}{2} \E\left[F(\x_{s-1}) - F_{\star}\right]$. Denoting $S = \mathcal{O}(\log (\frac{\epsilon_1}{\epsilon}))$ and $\epsilon_s = \epsilon_1 /2^{s}$, where $\epsilon_1$ is a positive constant, we can obtain the guarantees for optimal gap in the following theorem.

\begin{theorem}\label{thm:3}
Setting $B_1 = \mathcal{O}(1)$,  $\eta_s = \mathcal{O}(\epsilon_s^2)$, $\alpha_s = \mathcal{O}(\epsilon_s^2)$, and $T_s = \mathcal{O}(\epsilon_s^{-2})$, we ensure $\E\left[F(\x_{S})\right]-\min_{\hat{\x}} F(\hat{\x}) \leq \epsilon$ in $\mathcal{O}(\epsilon^{-2})$ iterations.
\end{theorem}
\textbf{Remark:} When using a constant batch size, we obtain $\mathcal{O}(\epsilon^{-2})$ complexity for both SFO and LMO, matching the results for single-level problems~\citep{Zhang2019OneSS}. 
Also note that the SFO complexity of $\mathcal{O}(\epsilon^{-2})$ is already optimal for convex objective functions~\citep{Agarwal2012InformationTheoreticLB}. 

\begin{theorem}\label{thm:3_0} (Large Batch) By setting $B_1=\mathcal{O}(\epsilon_s^{-1})$, $\eta_s = \mathcal{O}(\epsilon_s)$, $\alpha_s = \mathcal{O}(\epsilon_s)$, and $T_s = \mathcal{O}(\epsilon_s^{-1})$, we can ensure that $\E\left[F(\x_{S})\right]-\min_{\hat{\x}} F(\hat{\x}) \leq \epsilon$ in $\mathcal{O}(\epsilon^{-1})$ iterations.
\end{theorem}
\textbf{Remark:} The above theorem indicates that our algorithm achieves an SFO complexity of $\mathcal{O}(\epsilon^{-2})$ and an LMO complexity of $\mathcal{O}(\epsilon^{-1})$ with a large batch size of $\mathcal{O}(\epsilon^{-1})$, aligning with the rates of methods for single-level settings, such as Spider-FW~\citep{pmlr-v97-yurtsever19b}.

\paragraph{Strongly convex functions:} Compared with convex functions, we have an additional term $\frac{\lambda}{2} \Norm{\y-\x}^2$ for strongly convex objectives according to the Definition~\ref{def:sc}. So, instead of applying $\arg \min _{\x \in \X}\langle \x,\v_t\rangle$ in Step 9 of Algorithm~\ref{alg:1}, we aim to
\begin{align*}
    \min _{\w \in \X} g(\w,\x_t,\v_t) \coloneqq \left\{\langle \w,\v_t \rangle+\frac{\lambda}{4}\|\w-\x_t\|^{2}\right\}
\end{align*} 
via LMO in this case. To this end, we design a sub-algorithm that runs the Frank-Wolfe method~\citep{pmlr-v28-jaggi13} for several loops. In each loop $n$, we update as:
\begin{equation}
   \begin{split}
       \s  = &\arg \min _{\hat{\s} \in \X}\left\langle \v_t+\frac{\lambda}{2}\left(\w_n-\x_t\right), \hat{\s}\right\rangle, \\
&\w_{n+1}=\left(1-\gamma_{t}\right) \w_{n}+\gamma_{t} \s,
\end{split}
\end{equation}
where $\v_t+\frac{\lambda}{2}\left(\w_n-\x_t\right)$ is the gradient of $g(\w_n,\x_t,\v_t) $ with respect to parameter $\w_n$. Overall, we retain the stage-wise design of the original Algorithm 3, but replace step 9 in its basic Algorithm~\ref{alg:1} with the newly developed Algorithm~\ref{alg:multi-storm4}. This modification allows us to establish the optimal gap guarantees as detailed below. Here, we denote $\epsilon_s = \epsilon_1 /2^{s}$, where $\epsilon_1$ is a positive constant.
\begin{figure*}[ht]
\vskip 0.2in
	\begin{center}
		
			\includegraphics[width=0.37\textwidth]{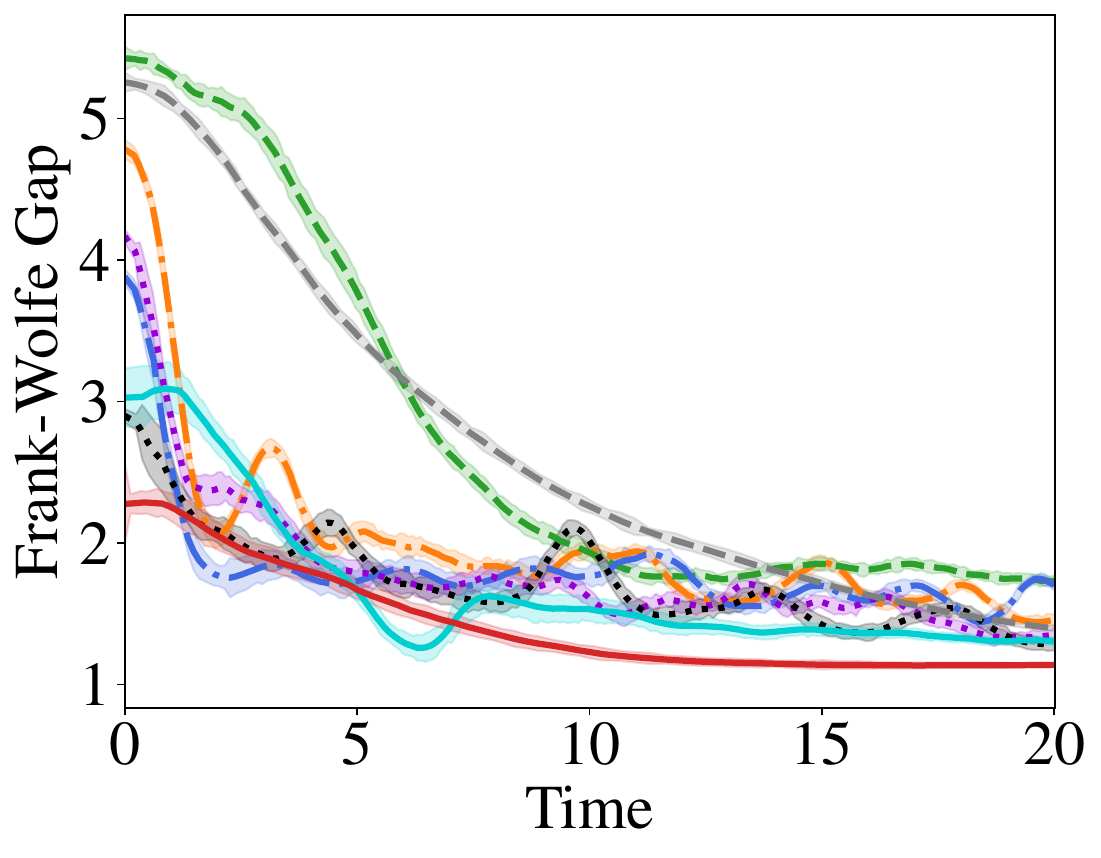}
			\hspace{0.5in}
   \includegraphics[width=0.37\textwidth]{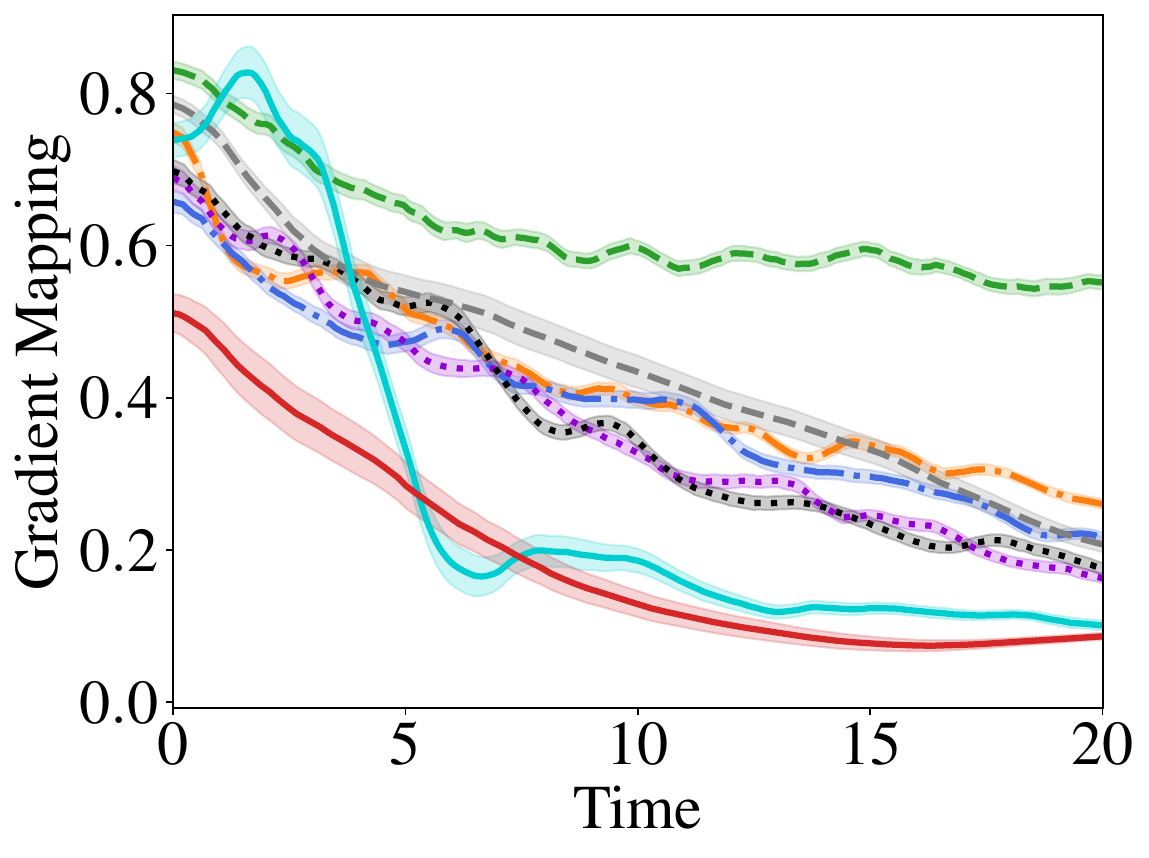}

		\subfigure{
			\includegraphics[width=0.99\textwidth]{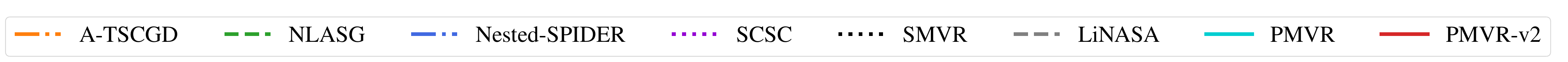}
		}

		\caption{Results for matrix optimization with low-rank constraints.}
		\label{fig:2}
	\end{center}
\vskip -0.2in
\end{figure*}

\begin{algorithm}[!tb]
	\caption{Stage-wise PMVR-v2}
	\label{alg:multi-storm4}
 Replace Step 9 of Algorithm~\ref{alg:1} with the following
 
	\begin{algorithmic}[1]
         \STATE Initialize $\w_1 = \x_t$
        \FOR{time step $n = 1$ {\bfseries to} $N$}
	\STATE Compute $ \s
      = \arg \min _{\hat{\s} \in \X}\left\langle \v_t+\frac{\lambda}{2}\left(\w_n-\x_t\right), \hat{\s}\right\rangle$
	\STATE Set $\w_{n+1}=\left(1-\gamma_{t}\right) \w_{n}+\gamma_{t} \s$
	\ENDFOR
     \STATE Set $\z_t = \w_{N+1}$
\end{algorithmic}
\end{algorithm}

\begin{theorem}\label{thm:4} 
Setting $B_1 = \mathcal{O}(1)$, $N=\mathcal{O}(\lambda \epsilon^{-1})$,  $\eta_s = \mathcal{O}(\lambda \epsilon_s)$, $\alpha_s = \mathcal{O}(\lambda \epsilon_s)$, and $T_s = \mathcal{O}(\lambda^{-1}\epsilon_s^{-1})$, we ensure $\E\left[F(\x_{S})\right]-\min_{\hat{\x}} F(\hat{\x}) \leq \epsilon$ in $S = \mathcal{O}(\log (\frac{\epsilon_1}{\epsilon}))$ stages.
\end{theorem}
\textbf{Remark:} When employing constant batch sizes, we can achieve an SFO complexity of $\mathcal{O}(\lambda^{-1}\epsilon^{-1})$ and an LMO complexity of $\mathcal{O}(\epsilon^{-2})$. Notably, the $\mathcal{O} (\lambda^{-1}\epsilon^{-1})$ SFO complexity we obtained is already optimal for stochastic unconstrained strongly convex problems~\citep{Agarwal2012InformationTheoreticLB} and is thus unimprovable.

\begin{theorem}\label{thm:4_0} (Large Batch)
By setting $B_1 = \mathcal{O}(\epsilon_s^{-1})$, $N= \mathcal{O}(\lambda \epsilon^{-1})$,  $\eta_s = \mathcal{O}(\lambda)$, $\alpha_s = \mathcal{O}(\lambda)$, and $T_s = \mathcal{O}(\lambda^{-1})$, we can guarantee that $\E\left[F(\x_{S})\right]-\min_{\hat{\x}} F(\hat{\x}) \leq \epsilon$ within $S = \mathcal{O}(\log (\frac{\epsilon_1}{\epsilon}))$ stages.
\end{theorem}

\textbf{Remark:} By using a large batch size, the SFO complexity remains on the same order, and the LMO complexity can be further improved to $\mathcal{O}(\epsilon^{-1})$, matching the existing results for strongly convex functions in the single-level setting~\citep{doi:10.1137/140992382}.

\section{Experiments}\label{sec:4}
In this section, we evaluate the effectiveness of our proposed methods through numerical experiments on three different problems. We compare our methods with existing stochastic multi-level algorithms, including A-TSCGD~\citep{Yang2019MultilevelSG}, NLASG~\citep{balasubramanian2020stochastic}, Nested-SPIDER~\citep{Zhang2021MultiLevelCS}, SCSC~\citep{chen2021solving}, and SMVR~\citep{jiang2022optimal}. We also compare with the previous projection-free multi-level algorithm LiNASA+ICG~\citep{NEURIPS2022_7e16384b}. For our algorithm, we select the momentum parameter $\alpha$ from the set $\{0.01, 0.03, 0.05, 0.1,0.3\}$ and search the parameter $N$ for PMVR-v2 from the range $\{10,50,100\}$. For the other methods, we adopt the hyper-parameters recommended in their original papers or perform a grid search to select the best ones. The learning rate is fine-tuned within the range of $\{0.001,0.005, 0.01,0.05,0.1\}$. All experiments are conducted on a personal laptop.

\subsection{Matrix Optimization with Low-Rank Constraints}
Following the previous literature on projection-free multi-level optimization~\cite{NEURIPS2022_7e16384b}, we also conduct experiments on matrix optimization with low-rank constraints. Specifically, we consider the matrix-valued single-index model~\cite{pmlr-v70-yang17a} with a low-rank constraint, expressed as:
\begin{gather*}
    y=\left|\langle A, B^{\star}\rangle_{F}\right|^{2} + \epsilon, \quad
    \text{rank}(B^{\star}) \leq s,
\end{gather*}
where $A,B \in \R^{m\times n}$, $\epsilon \sim \mathcal{N}\left(0, \sigma^{2}\right),\langle\cdot, \cdot\rangle_{F}$ denotes the Frobenius inner product, and $s$ is a positive integer smaller than both $m$ and $n$. To recover a low-rank matrix $B^\star$ given $A$ and $y$, we can optimize the mean squared loss function with a nuclear norm constraint.  The objective function can be formulated as:
\begin{align*}
    \min_{B} F(B)&=\E\left[\left(y-\left|\langle A, B\rangle_{F}\right|^{2}\right)^{2}\right]\\
    &\text{s.t. } \ \|B\|_{\star} \leq s.
\end{align*}
Note that in this case, the projection operation onto the nuclear norm ball requires a full singular value decomposition~(SVD), while the linear optimization used by Frank-Wolfe update only requires computing the singular vector pair corresponding to the largest singular value, which is much cheaper~\cite{pmlr-v28-jaggi13}. In line with the setup in~\citet{NEURIPS2022_7e16384b}, we define the matrix $B^{\star}= v v^{\top} /\left\|v v^{\top}\right\|_{\star}$, and the matrix $A$ is generated by $A=I+E$, with $E_{i, j} \stackrel{i . i . d .}{\sim} \mathcal{N}\left(0, 0.3\right)$. 

In Figure~\ref{fig:2}, we plot the value of Frank-Wolfe gap, as well as gradient mapping, against the running time of each algorithm. All the curves are averaged over 50 runs. As can be seen, our PMVR and PMVR-v2 methods demonstrate a more rapid decrease in both criteria compared to other approaches, especially for the gradient mapping criterion, demonstrating the superiority of our proposed methods.

\begin{figure*}[t]
\vskip 0.2in
	\begin{center}
		\subfigure{
			\includegraphics[width=0.32\textwidth]{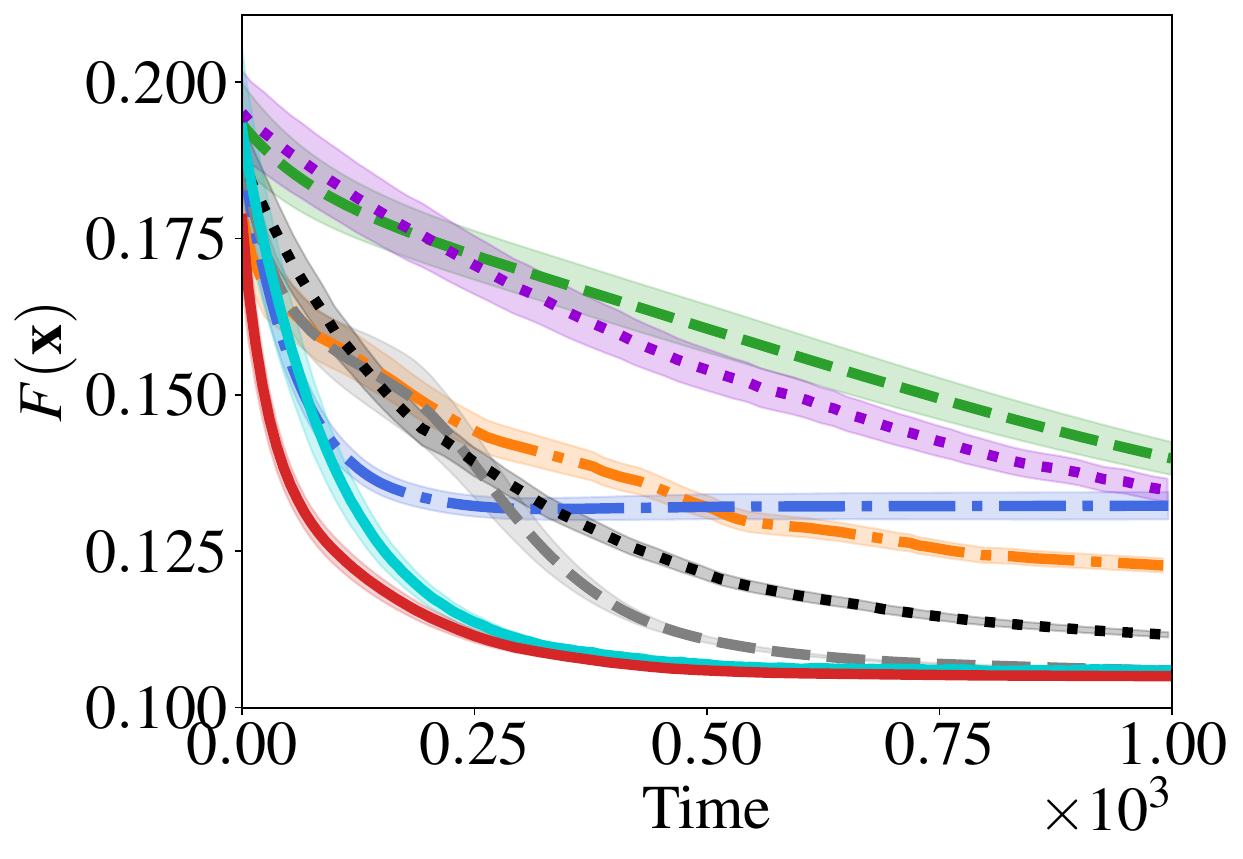}}
   \setcounter{subfigure}{0}
   \subfigure[ Industry-10]{
			\includegraphics[width=0.32\textwidth]{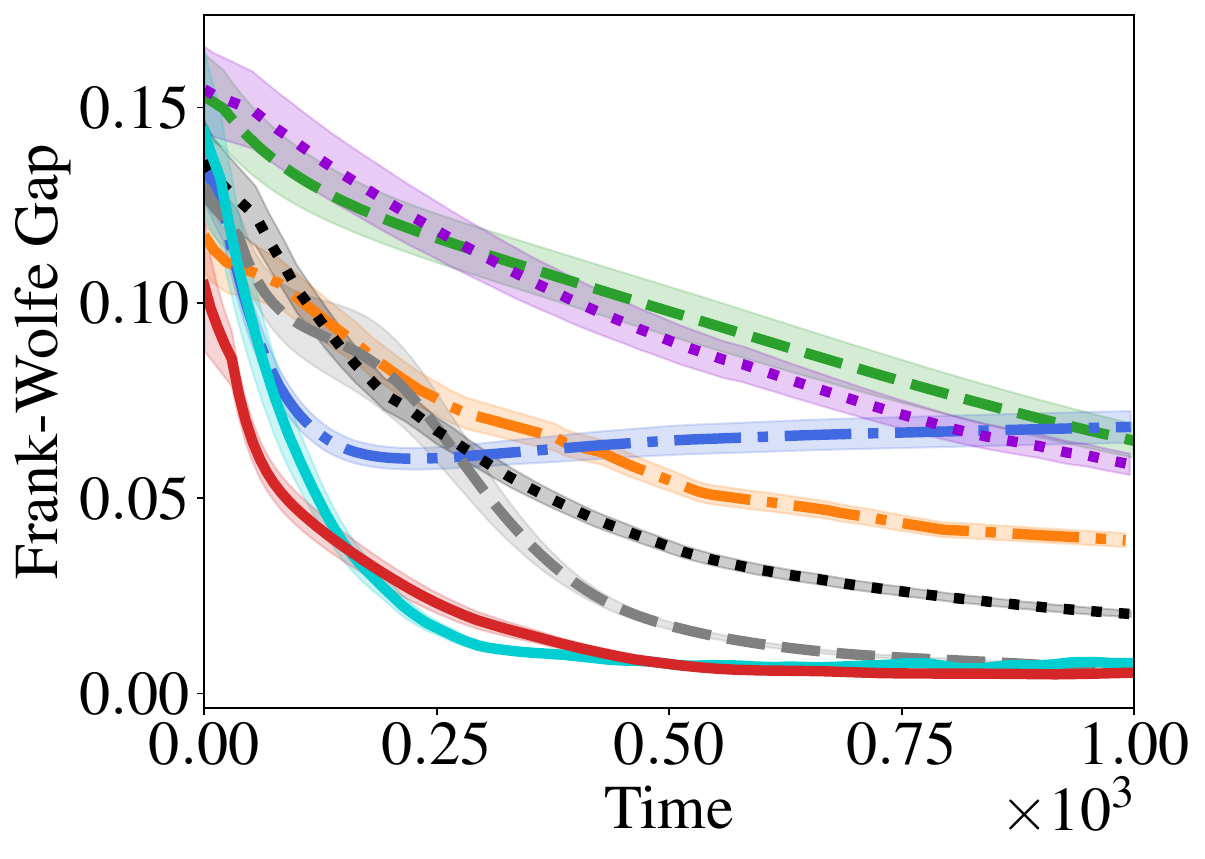}}
   \subfigure{
			\includegraphics[width=0.32\textwidth]{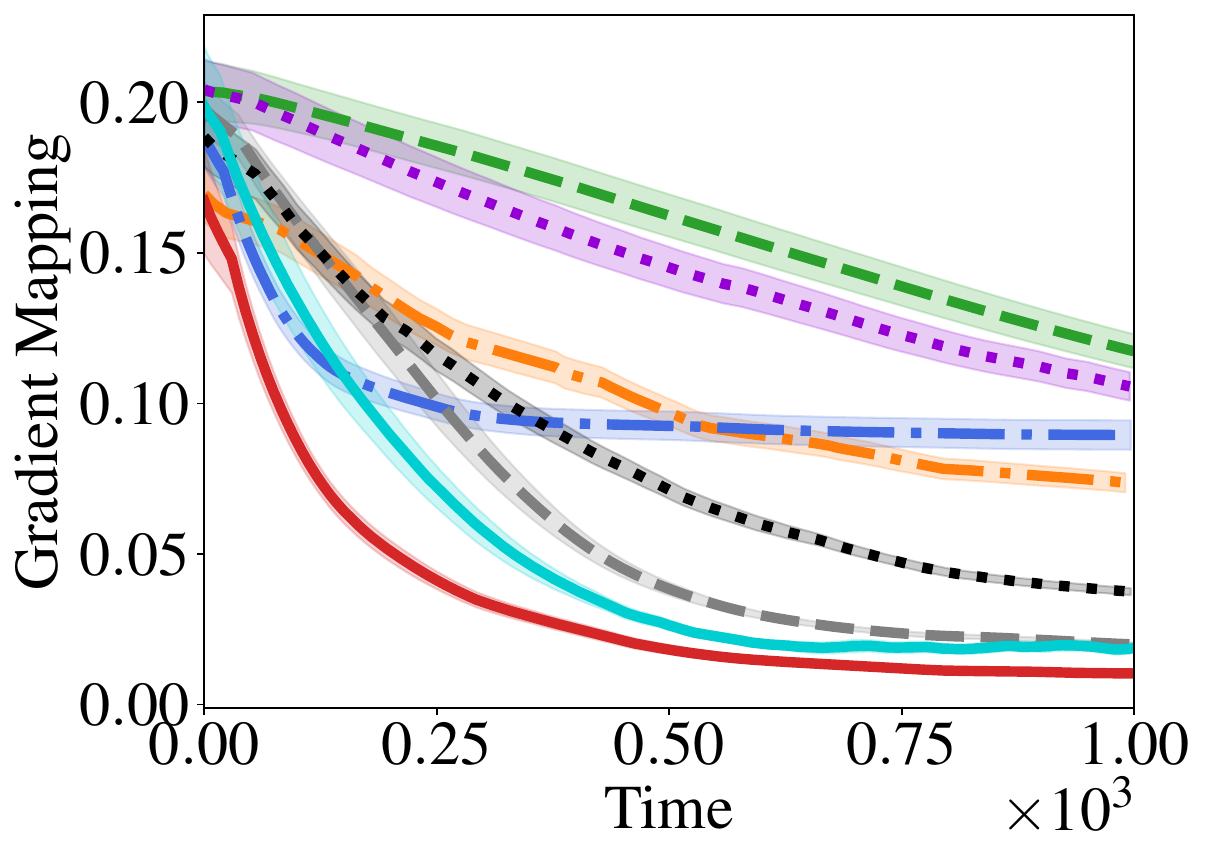}
		}
  \setcounter{subfigure}{0}
		\subfigure{
			\includegraphics[width=0.32\textwidth]{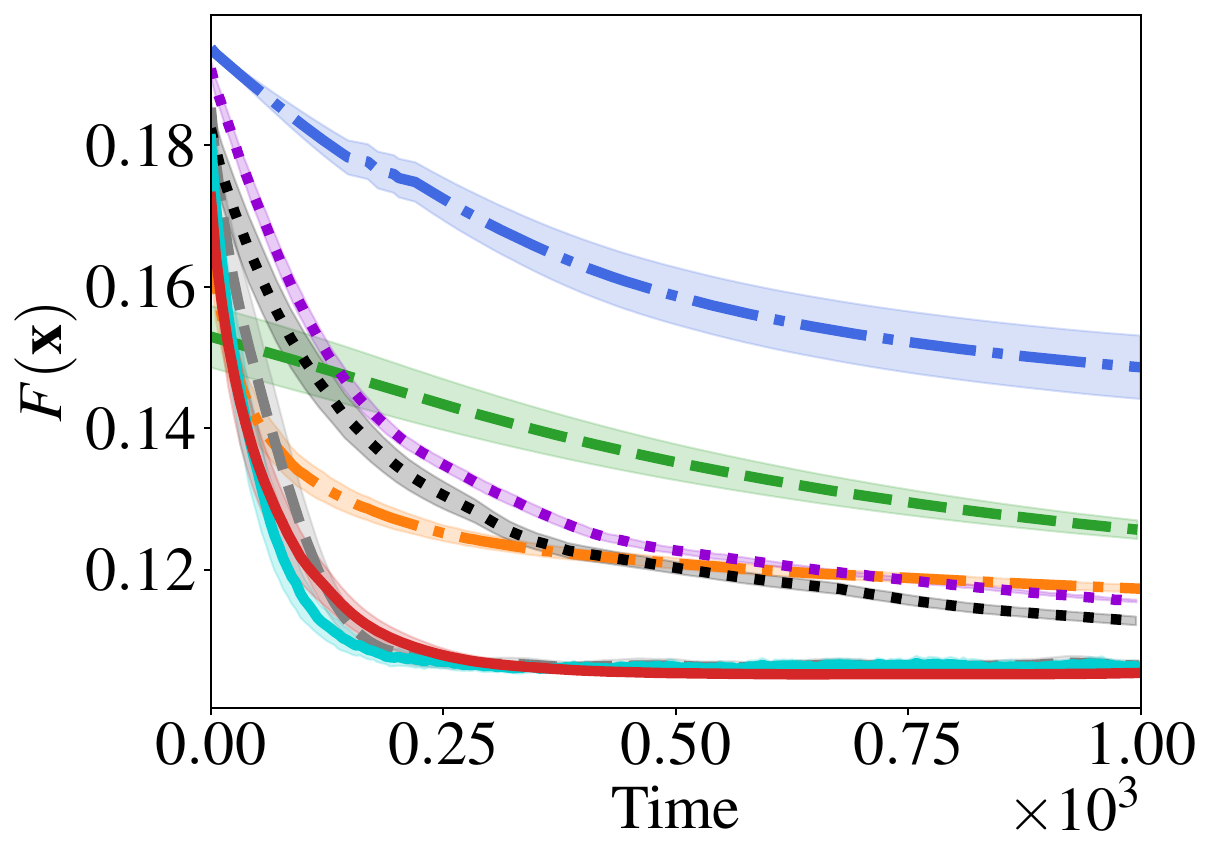}}
   \subfigure[ Industry-12]{
			\includegraphics[width=0.32\textwidth]{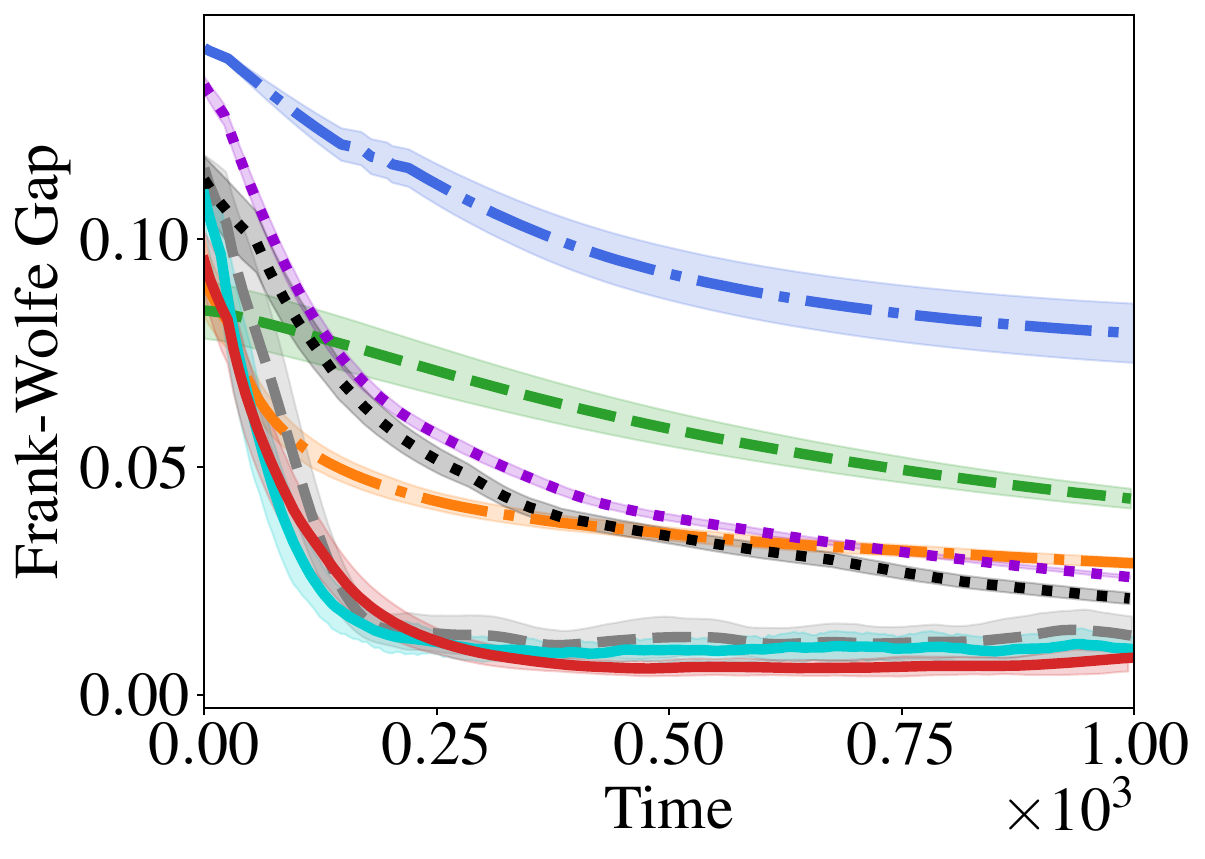}}
   \subfigure{
			\includegraphics[width=0.32\textwidth]{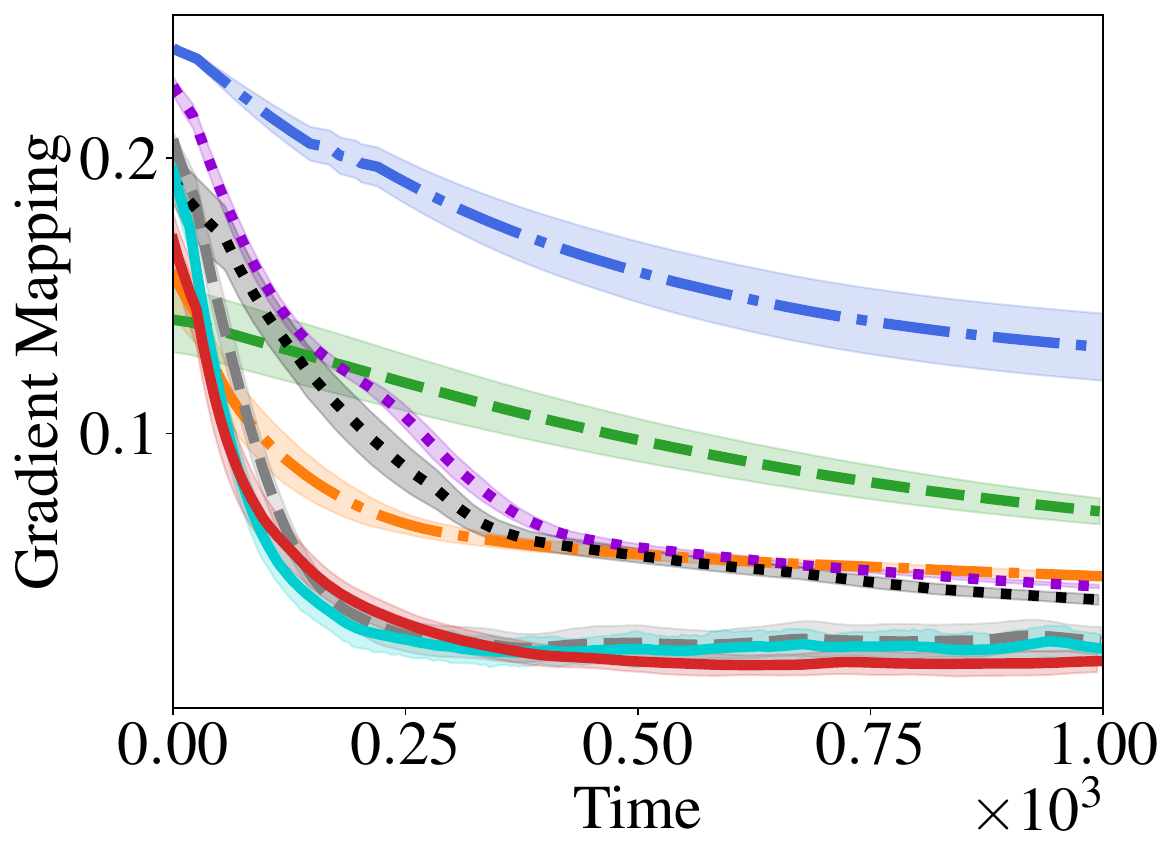}
		}
		\subfigure{
			\includegraphics[width=0.99\textwidth]{./Fig/legend.pdf}
		}
		\caption{Results for risk-averse portfolio optimization.}
		\label{fig:1}
	\end{center}
\vskip -0.2in
\end{figure*}

\begin{figure*}[t]
\vskip 0.2in
	\begin{center}
		\subfigure{
			\includegraphics[width=0.32\textwidth]{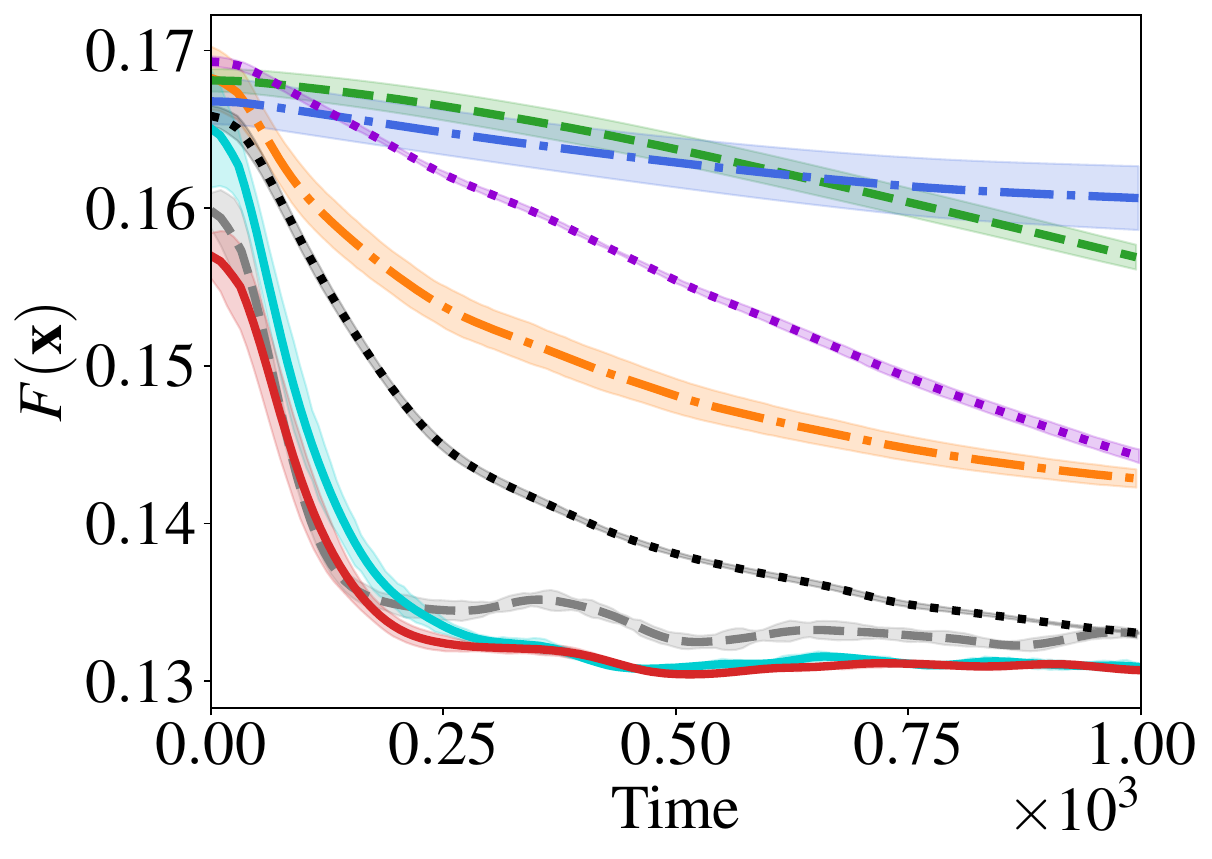}}
   \setcounter{subfigure}{0}
   \subfigure[ Industry-10]{
			\includegraphics[width=0.32\textwidth]{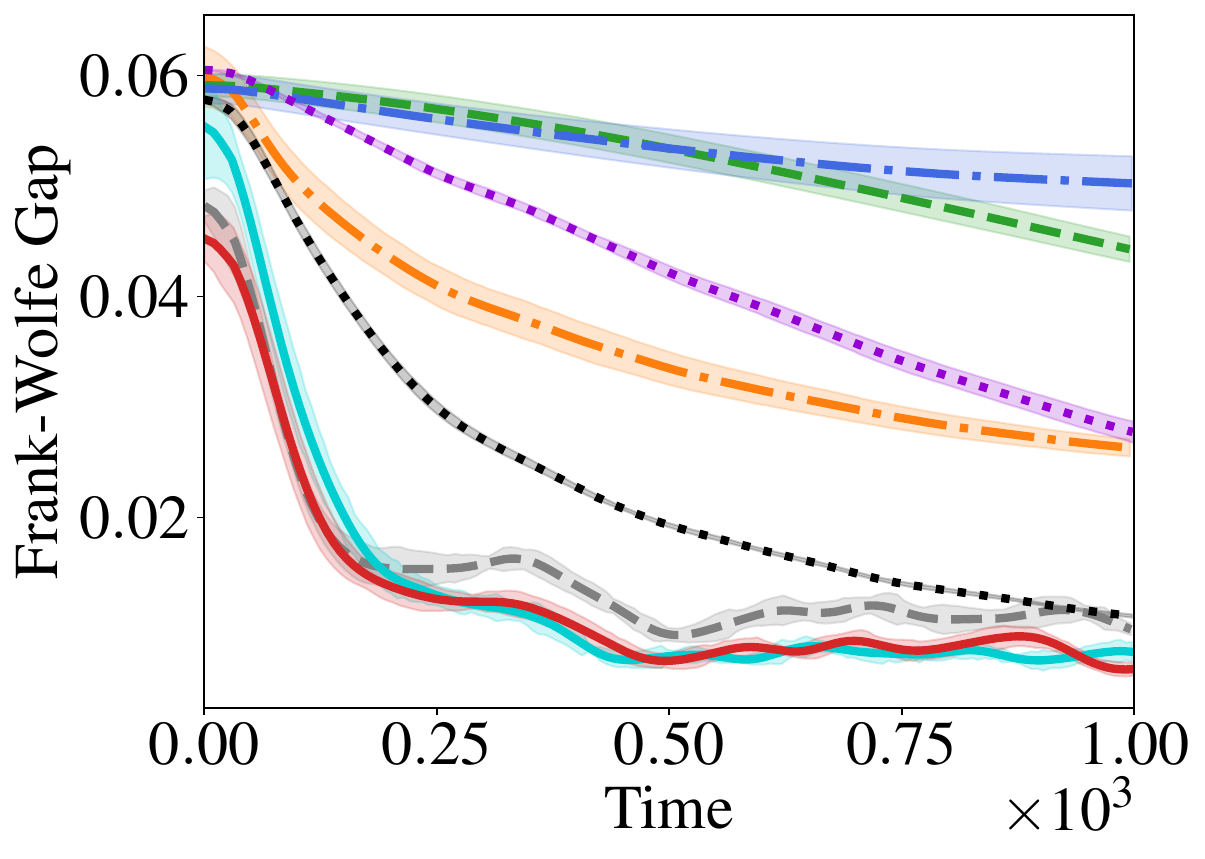}}
   \subfigure{
			\includegraphics[width=0.32\textwidth]{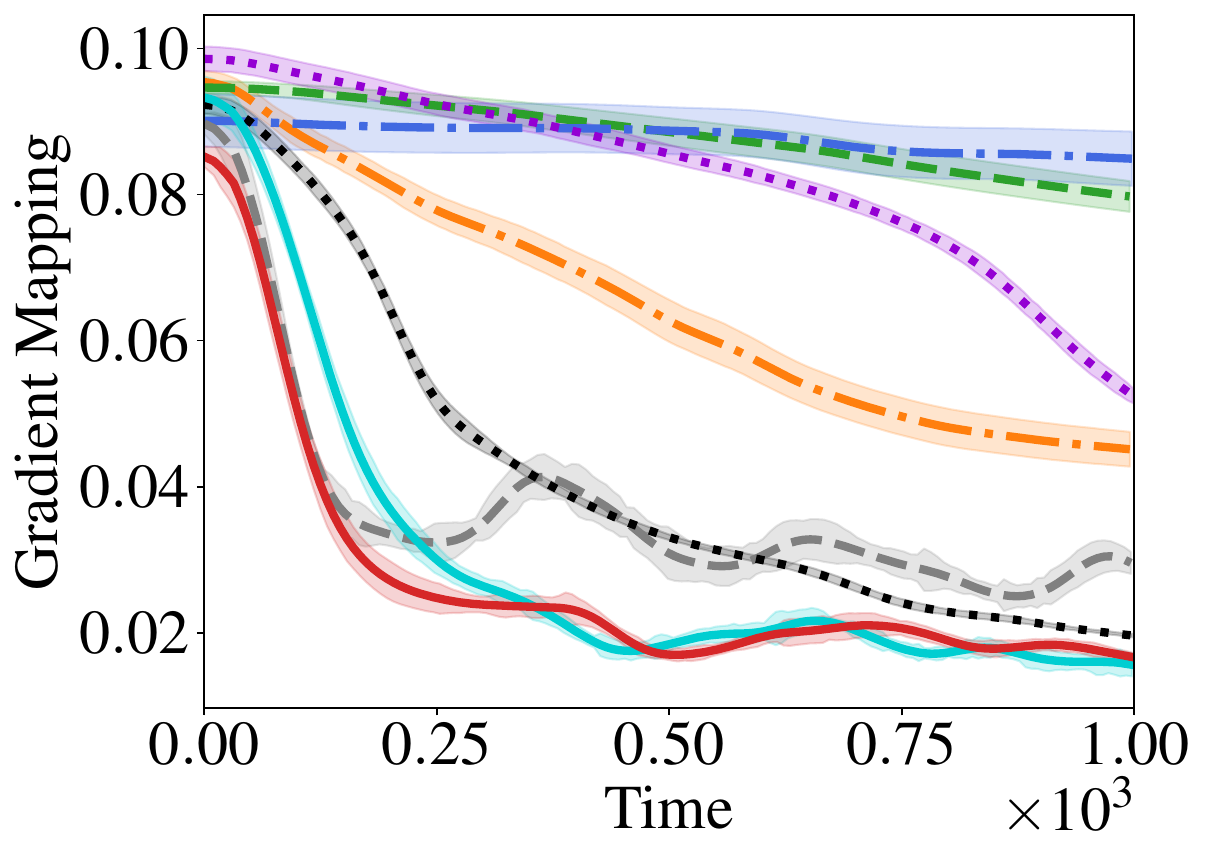}
		}
  \setcounter{subfigure}{0}
		\subfigure{
			\includegraphics[width=0.32\textwidth]{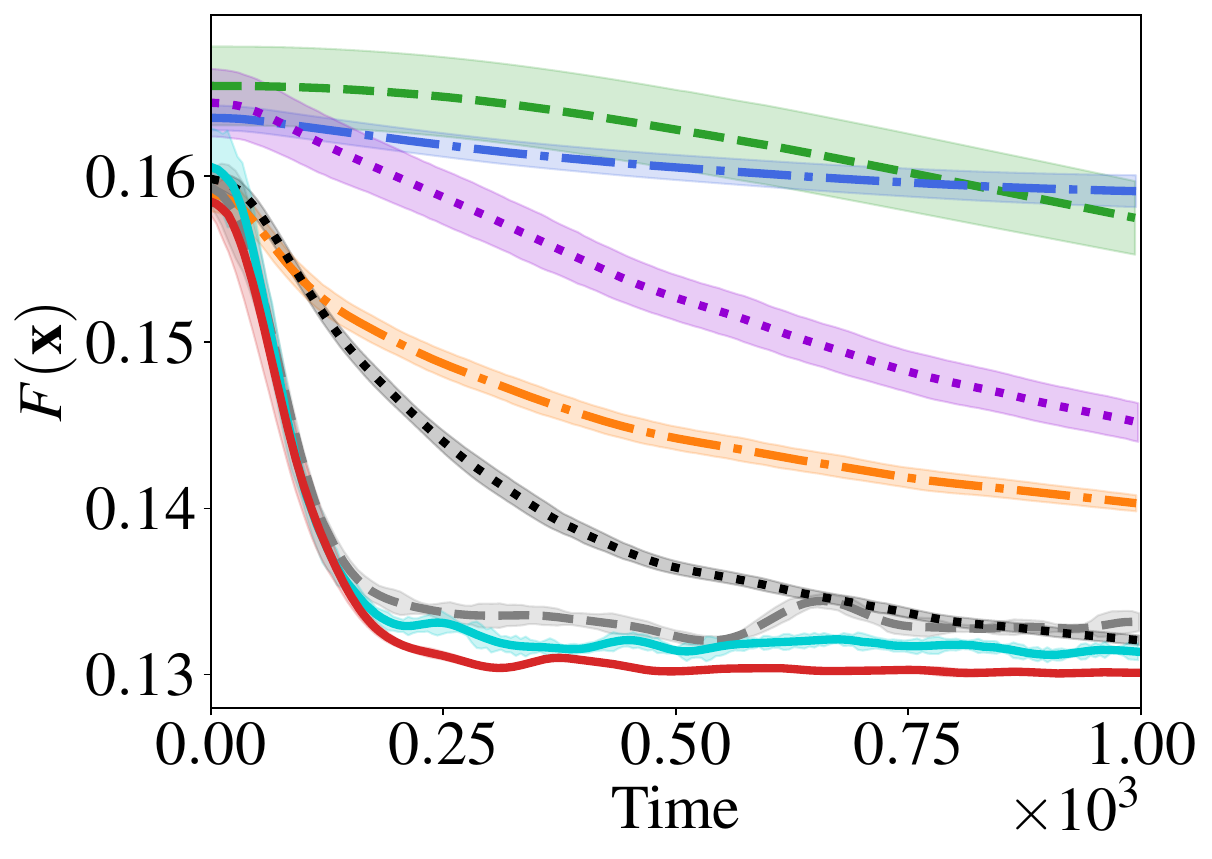}}
   \subfigure[ Industry-12]{
			\includegraphics[width=0.32\textwidth]{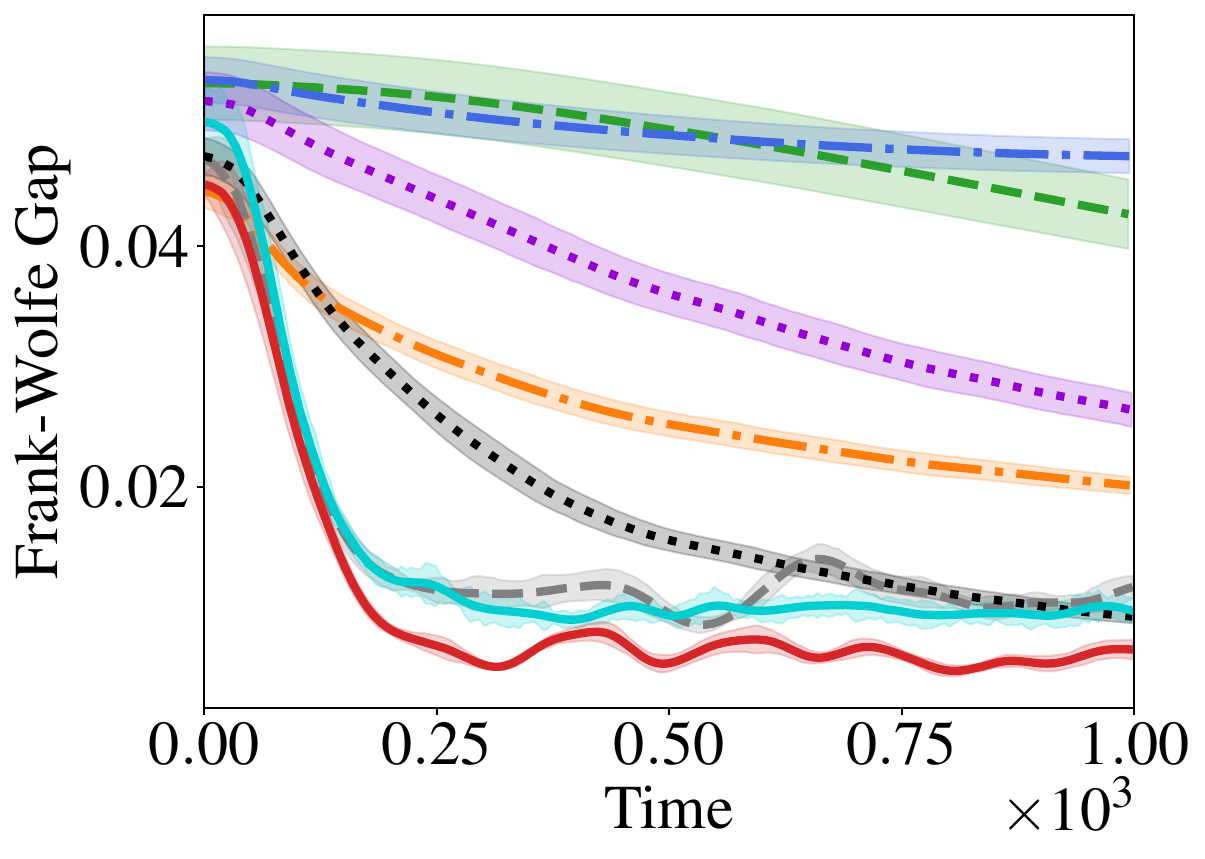}}
   \subfigure{
			\includegraphics[width=0.32\textwidth]{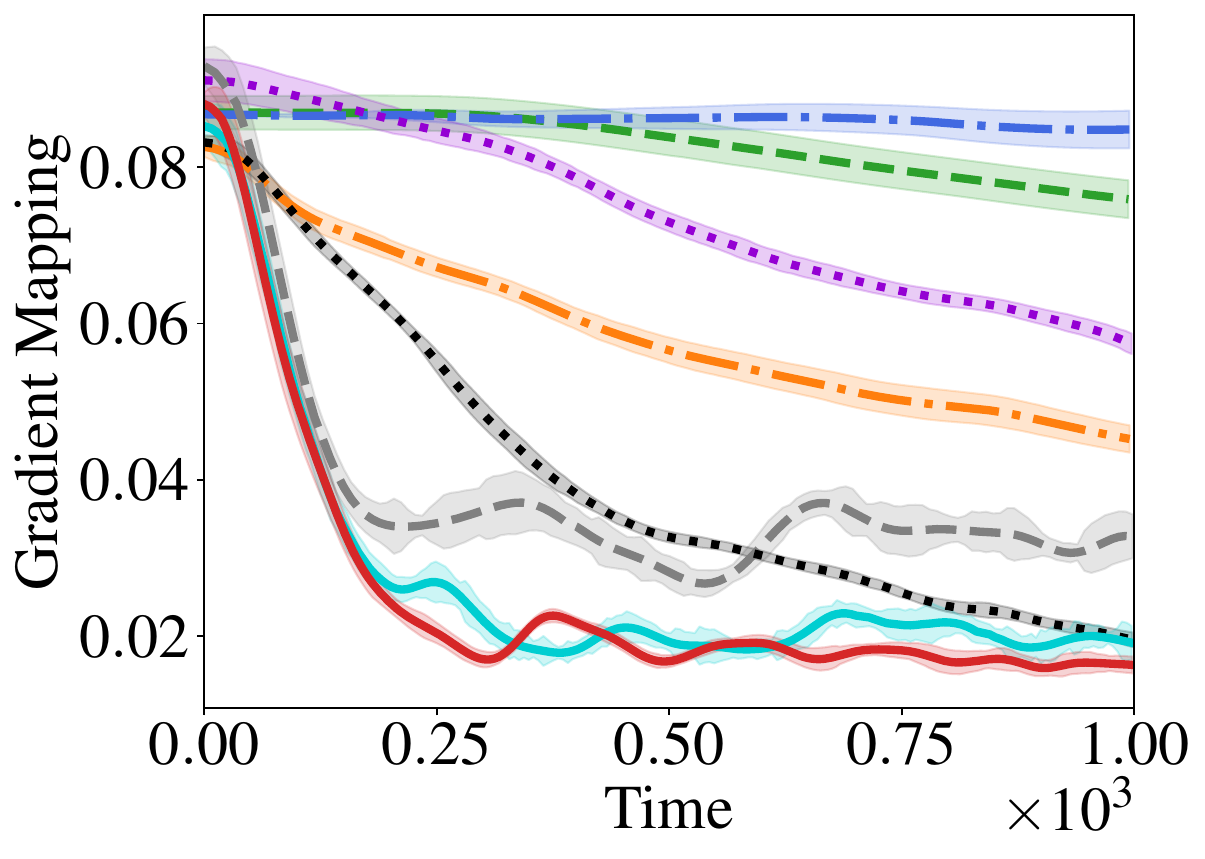}
		}
		\subfigure{
			\includegraphics[width=0.99\textwidth]{./Fig/legend.pdf}
		}
		\caption{Results for risk-averse portfolio optimization.}
		\label{fig:3}
	\end{center}
\vskip -0.2in
\end{figure*}
\subsection{Mean-variance Risk-averse Optimization}
We first consider the problem of risk-averse portfolio optimization~\citep{Shapiro2009LecturesOS}, a commonly used benchmark in comparing various multi-level algorithms~\citep{Yang2019MultilevelSG,balasubramanian2020stochastic,Zhang2021MultiLevelCS,chen2021solving}. Suppose we have $d$ assets to invest over time steps $\{1, \ldots, T\}$, and $\r_{t} \in \mathbb{R}^{d}$ represents the payoff of $d$ assets at time step $t$. The goal is to maximize investment returns and minimize risk simultaneously. A suitable approach for this purpose is the mean-variance risk-averse optimization model, where risk is defined as the variance. This optimization problem is formulated as:
\begin{gather*}
\min _{\x \in \mathcal{X}} F(\x) =  -\frac{1}{T} \sum_{t=1}^{T}\left\langle \r_{t}, \x\right\rangle+ \frac{\lambda}{T} \sum_{t=1}^{T}\left(\left\langle \r_{t}, \x\right\rangle-\left\langle \bar{\r}, \x\right\rangle\right)^{2},
\end{gather*}
where $\bar{\r} = \frac{1}{T}\sum_{t=1}^{T} \r_{t}$, and the decision variable $\x$ denotes the investment quantities in $d$ assets. The domain $\X$ is a simplex, ensuring that $\Norm{\x}_{1}=1$ for any $\x \in \X$. This problem can be modeled as a stochastic two-level constrained compositional optimization problem,  with each layer expressed as follows: 
\begin{gather*}
     f_{1}(\x)=\left(-\frac{1}{T} \sum_{t=1}^{T}\langle \r_{t}, \x\rangle, \x\right), \\
     f_{2}(\y_1, \y_2)=\y_1 + \frac{\lambda}{T} \sum_{t=1}^{T}\left(\left\langle \r_{t}, \y_2\right\rangle+\y_1\right)^{2}, 
\end{gather*}
where $f_1(\cdot)$ is the inner function and $f_2(\cdot)$ is the outer function such that $F(\x)=f_2(f_1(\x))$.

For experimental validation, we utilize real-world datasets Industry-10 and Industry-12 from the Kenneth R. French Data Library\footnote{https://mba.tuck.dartmouth.edu/pages/faculty/ken.french/}. These datasets include payoffs for 10 and 12 industrial assets over 25,105 consecutive periods. For non-projection-free methods, we implement the projection onto the simplex following a well-known efficient projection method~\cite{Duchi2008EfficientPO}. We report the loss value $F(\x)$, as well as the Frank-Wolfe gap and gradient mapping criteria in Figure~\ref{fig:1}, averaging all curves over 50 runs. We also include results concerning the number of iterations in Figure~\ref{fig:4} in the Appendix~\ref{EPAP}. It is observed that LiNASA+ICG, our PMVR and PMVR-v2, tend to converge more rapidly compared to other algorithms in  all tasks. Specifically, PMVR demonstrates similar performance to LiNASA+ICG in the Industry 12 dataset and performs better than LiNASA+ICG in Industry 10. The loss value, Frank-Wolfe gap, and gradient mapping of our PMVR-v2 decrease most quickly in both datasets, validating the effectiveness of the proposed method.

\subsection{Mean-deviation Risk-averse Optimization}
Finally, we further conduct the experiment on a three-level compositional optimization problem. Here, we still consider the problem of risk-averse portfolio optimization as in the previous subsection, and the risk  is quantified as the standard deviation this time. The mathematical formulation of this problem is:

\begin{align*} \max_{\mathbf{x} \in \mathcal{X}} \frac{1}{T} \sum_{t=1}^{T}\left\langle \mathbf r_t, \mathbf x \right\rangle-\lambda \sqrt{\frac{1}{T} \sum_{t=1}^{T}\left(\left\langle \mathbf r_{t}, 
\mathbf x \right\rangle-\left\langle \mathbf{\bar r}, \mathbf x \right\rangle\right)^{2}}, \end{align*}

where $\mathbf{\bar r}= \frac{1}{T}\sum_{t=1}^{T} \mathbf r_{t}$, $\mathcal{X}$ is the probability simplex, and the decision variable $ \mathbf x$ denotes the investment quantity vector in the $d$ assets. This is a three-level compositional optimization problem according to the analysis by \citet{jiang2022optimal}.

In the experiments, we also evaluated different methods on the real-world datasets Industry-10 and Industry-12. The results, including the loss value $F(\mathbf x)$, the Frank-Wolfe gap, and the gradient mapping criteria, are reported in Figure~\ref{fig:3}. These results are averaged over 10 runs. As can be seen, our methods (PMVR and PMVR-v2) demonstrated faster convergence compared to other algorithms across all tasks. The loss value, Frank-Wolfe gap, and gradient mapping of our PMVR-v2 decreased most rapidly in both datasets, indicating the effectiveness of our proposed approach.

\section{Conclusion}
In this paper, we investigate projection-free algorithms for stochastic constrained multi-level optimization. Our proposed methods not only yield better results than previous approaches under the gradient mapping criterion but also provide guarantees for the Frank-Wolfe gap, an aspect previously absent in projection-free multi-level research. Additionally, we provide theoretical guarantees for convex and strongly convex objective functions, and validate the effectiveness of our proposed methods through numerical experiments.

\section*{Acknowledgements}
This work was partially supported by NSFC (62122037, 61921006), the Collaborative Innovation Center of Novel Software Technology and Industrialization, and the Postgraduate Research \& Practice Innovation Program of Jiangsu Province (No. KYCX24\_0231).

\section*{Impact Statement}
This paper presents work whose goal is to advance the field of Machine Learning. There are many potential societal consequences of our work, none of which we feel must be specifically highlighted here.
\bibliography{ref}
\bibliographystyle{icml2024}

\newpage
\appendix
\onecolumn
\section{More Experimental Results}\label{EPAP}
We also report the loss value $F(\x)$, as well as the Frank-Wolfe gap and gradient mapping criteria concerning the number of iterations for the mean-variance risk-averse optimization in Figure~\ref{fig:4}. When the cost of the projection operation is not considered~(since we report the results based on iteration numbers rather than the time used), the SMVR, LiNASA+ICG, our PMVR and PMVR-v2, tend to converge more rapidly compared to other algorithms in all tasks. The loss value, Frank-Wolfe gap, and gradient mapping of our PMVR-v2 decrease most quickly in both datasets, validating the effectiveness of the proposed method.
\begin{figure*}[ht]
	\begin{center}
		\subfigure{
			\includegraphics[width=0.32\textwidth]{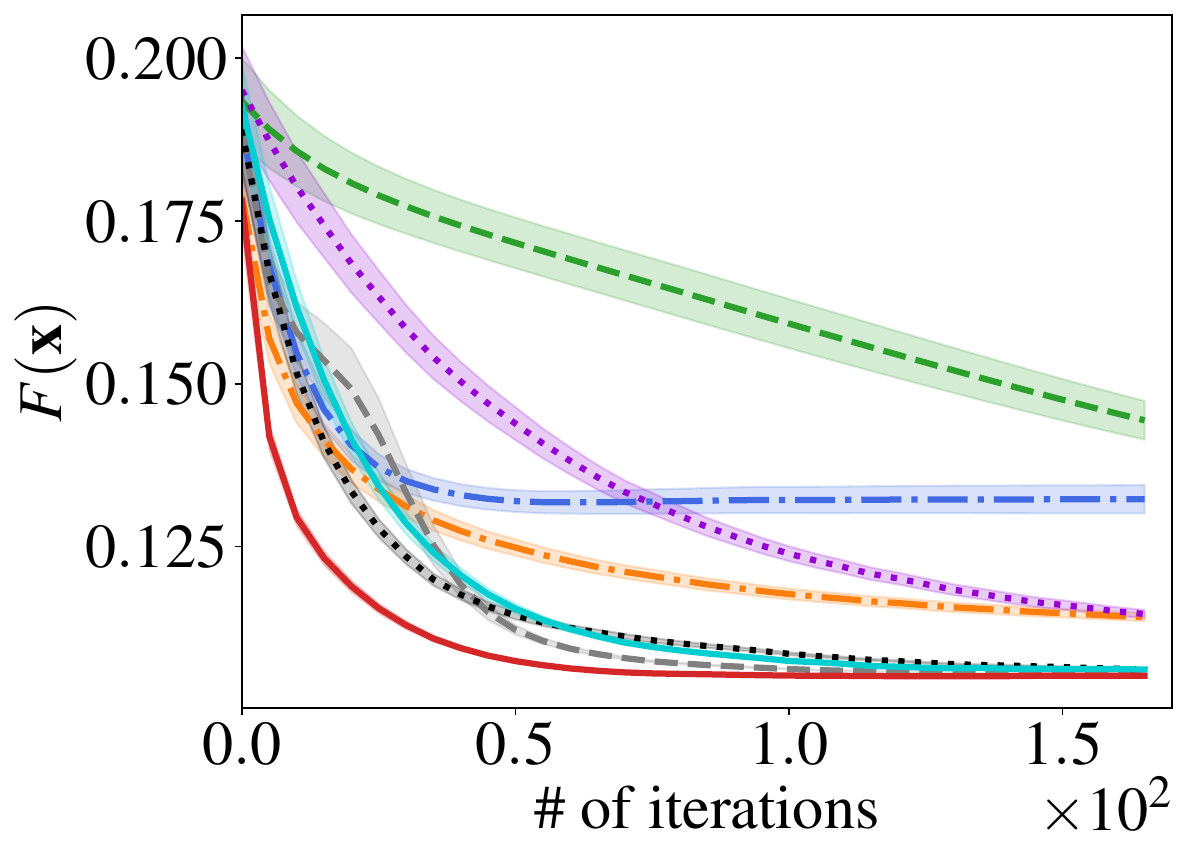}}
   \setcounter{subfigure}{0}
   \subfigure[ Industry-10]{
			\includegraphics[width=0.32\textwidth]{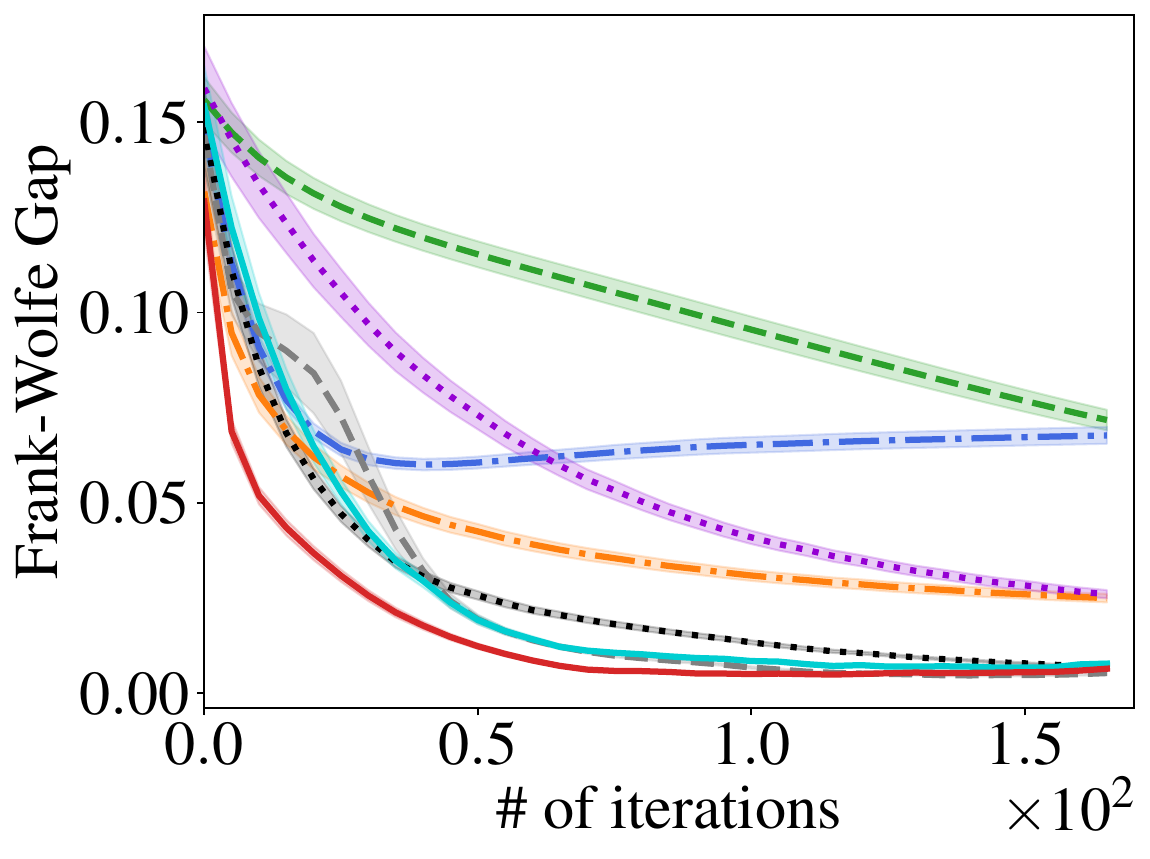}}
   \subfigure{
			\includegraphics[width=0.32\textwidth]{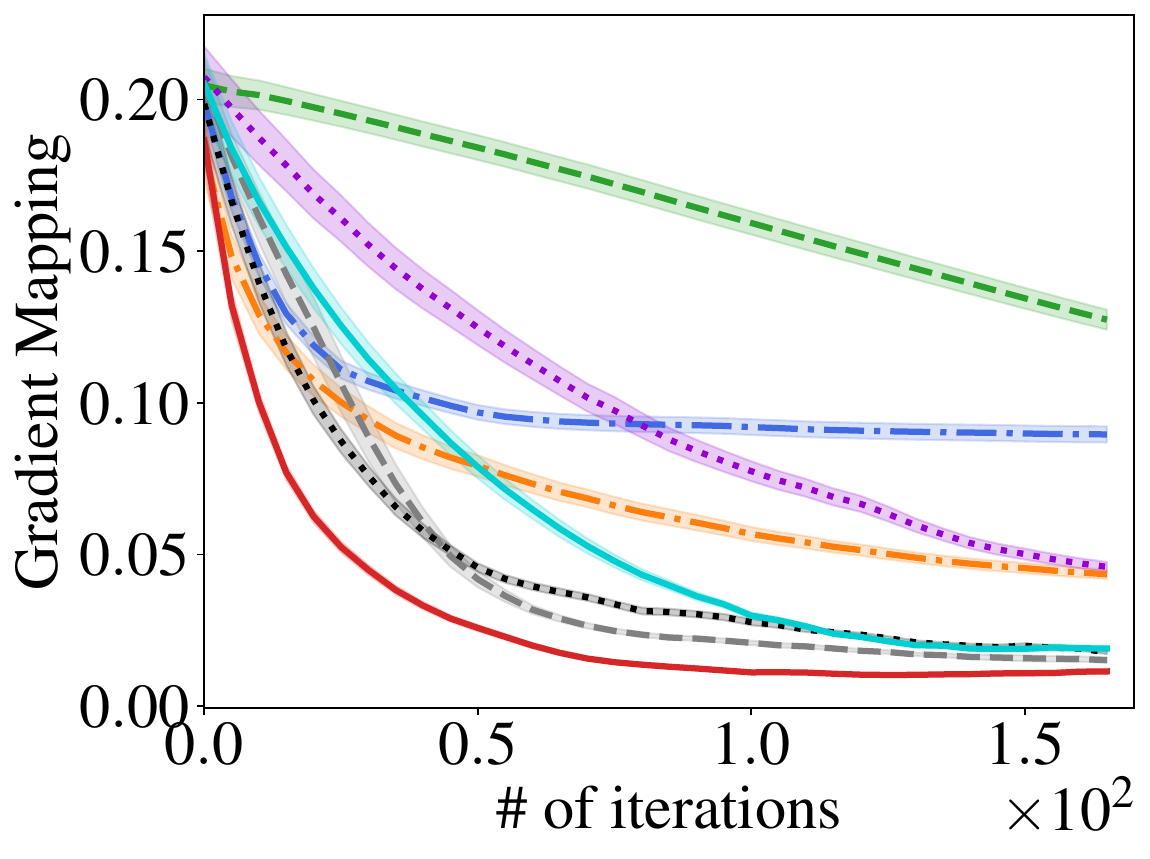}
		}
  \setcounter{subfigure}{0}
		\subfigure{
			\includegraphics[width=0.32\textwidth]{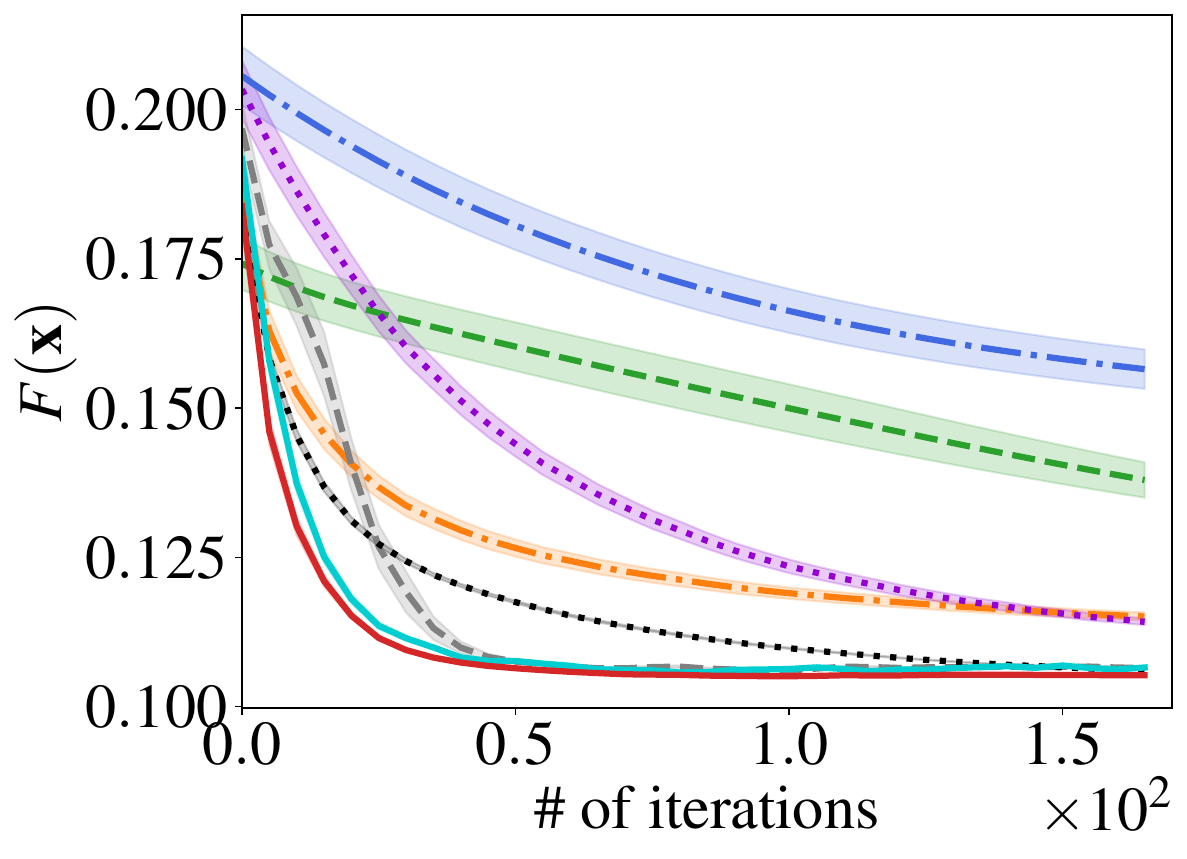}}
   \subfigure[ Industry-12]{
			\includegraphics[width=0.32\textwidth]{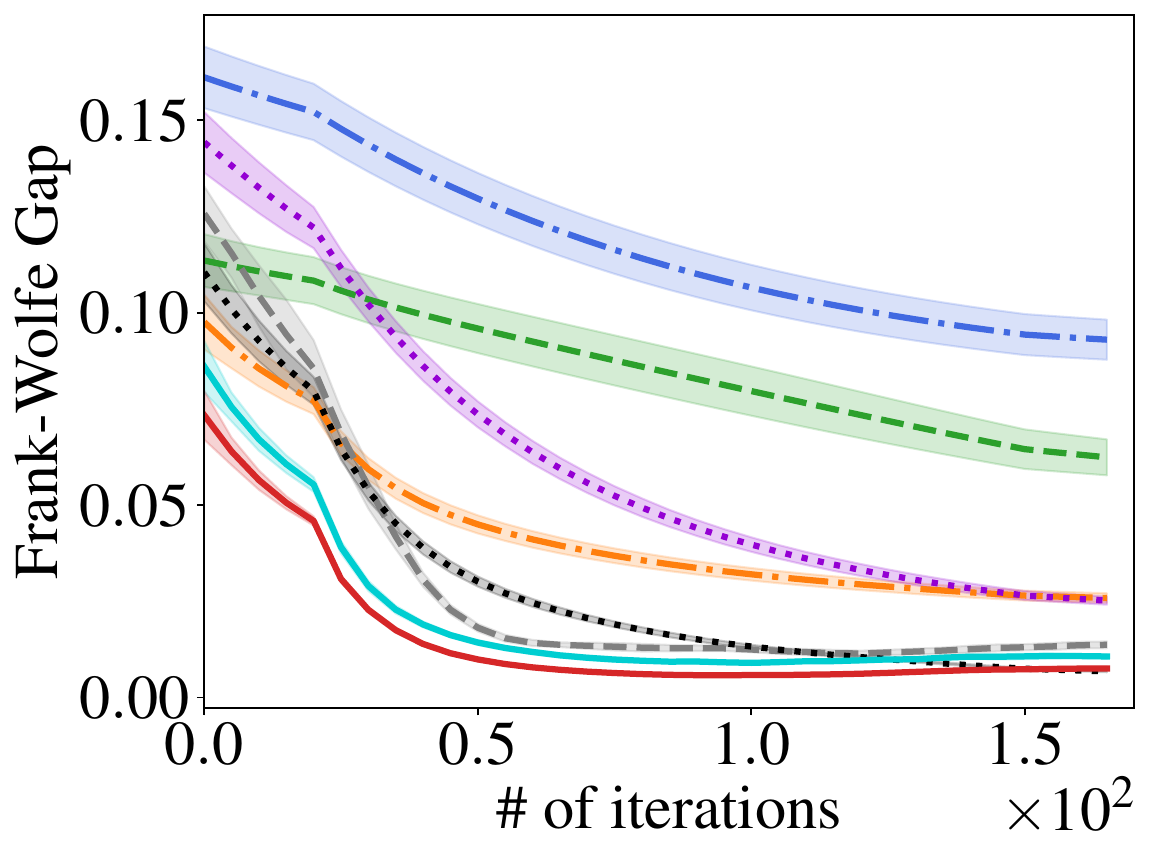}}
   \subfigure{
			\includegraphics[width=0.32\textwidth]{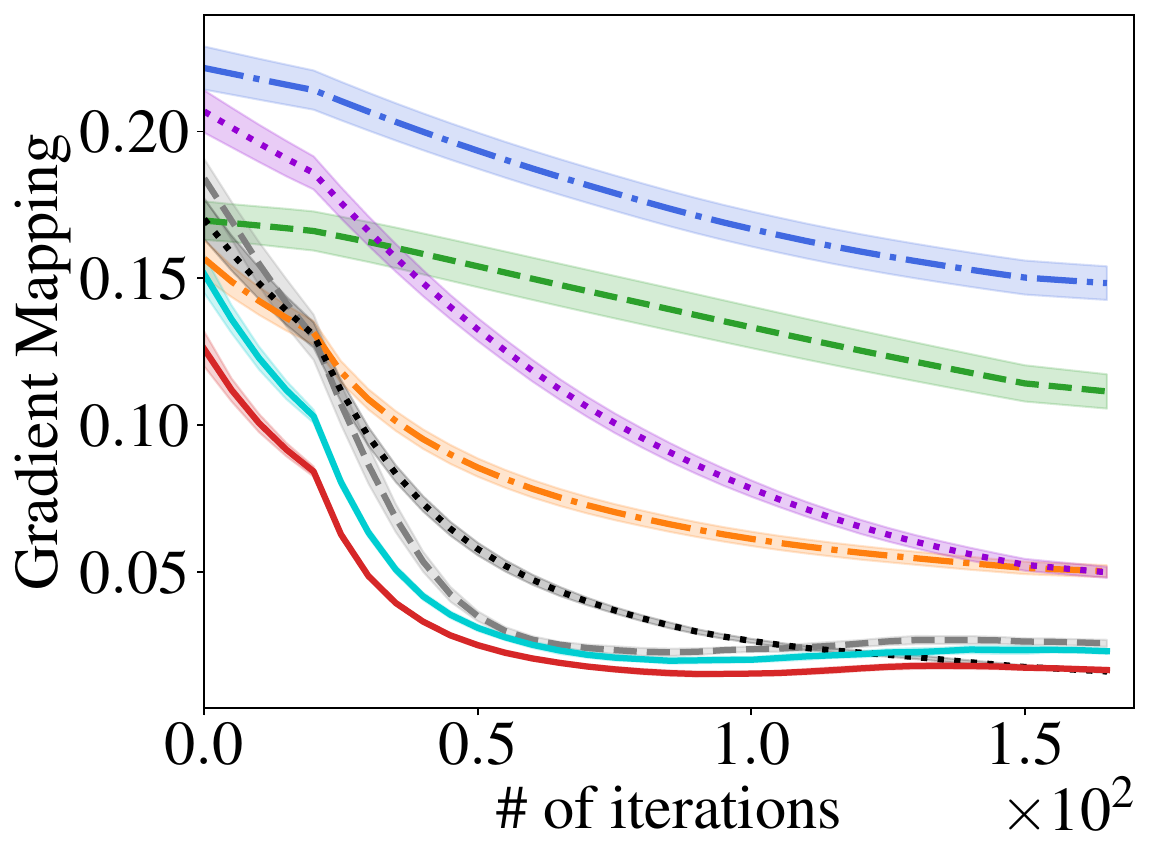}
		}
		\subfigure{
			\includegraphics[width=0.99\textwidth]{./Fig/legend.pdf}
		}
		\caption{Results for risk-averse portfolio optimization.}
		\label{fig:4}
	\end{center}
\end{figure*}

\section{Proof of Frank-Wolfe Gap~(Theorem~\ref{thm:main} and Theorem~\ref{thm:main_0})}
In this section, we present the proofs for the Frank-Wolfe gap. First, to bound the estimation error of the gradient, we have the following lemma.
\begin{lemma}\label{lem:2} The gradient estimation error can be divided into the following two parts.
\begin{align*}
\sum_{t=1}^T \E\left[\Norm{\nabla F(\x_t) - \v_t }^2\right] \leq  2 \sum_{t=1}^T\E\left[\Norm{\v_t - \prod_{i=1}^K \nabla f_i(\u_t^{i-1}) }^2\right] + 2KL_F^2 \sum_{t=1}^T \sum_{i=1}^{K-1} \E\left[\Norm{\u_t^i - f_i(\u_t^{i-1})}^2\right].
\end{align*}
\end{lemma}	
\begin{proof}
First, it is easy to show that:
\begin{align*}
\E\left[\Norm{\v_t - \nabla F(\x_t)}^2\right] \leq & 2\E\left[\Norm{\v_t - \prod_{i=1}^K\nabla f_i(\u_t^{i-1})}^2\right]  + 2\E\left[\Norm{\prod_{i=1}^K\nabla f_i(\u_t^{i-1}) - \nabla F(\x_t)}^2\right].
\end{align*}

Then, we define that $\y_t^i = f_i\circ f_{i-1}\circ \dotsc \circ f_1(\w_t)$ and $\nabla \widehat{F}_i(\w_t) = \nabla f_1(\w_t) \cdots \nabla f_i(\u^{i-1}_t)$.  	
	\begin{align*}
		\Norm{\nabla F_1(\w_t) - \nabla \widehat{F}_1(\w_t)} &= 0,\\
		\Norm{\nabla F_2(\w_t) - \nabla \widehat{F}_2(\w_t)} &= \Norm{\nabla f_1(\w_t) \nabla f_2(\y^1_t) - \nabla f_1(\w_t) \nabla f_2(\u^1_t)}\\
		&\leq  L_f L_J\Norm{\y_t^1 - \u^1_t},	\\
  \Norm{\nabla F_3(\w_t) - \nabla \widehat{F}_3(\w_t)} &= \Norm{\nabla f_1(\w_t) \nabla f_2(\y^1_t) \nabla f_3(\y^2_t) -  \nabla f_1(\w_t) \nabla f_2(\u^1_t) \nabla f_3(\u^2_t)}\\
		&\leq L_f^2 L_J\left(\Norm{\y_t^2 - \u_t^2}  +  \Norm{\y_t^1 - \u^1_t}\right),\\
		& \cdots\\
		\Norm{\nabla F_K(\w_t) - \nabla\widehat{F}_K(\w_t)}  &\leq L_f^{K-1}L_J \sum_{i=1}^{K-1} \Norm{\y_t^i - \u_t^i},
	\end{align*}
	Besides, we also have:
	\begin{align*}
		\Norm{\y_t^2 - \u_t^2} &=\Norm{f_2\circ f_1(\w_t) - \u^2_t} \\
  &\leq \Norm{f_2\circ f_1(\w_t) - f_2(\u^1_t)} + \Norm{f_2(\u^1_t)- \u^2_t}\\
		& \leq L_f \Norm{f_1(\w_t) - \u^1_t} + \Norm{f_2(\u^1_t)- \u^2_t},\\
		\Norm{\y_t^3-\u_t^3} &= \Norm{f_3\circ f_2\circ f_1(\w_t) - \u^3_t} \\
  &\leq \Norm{f_3\circ f_2 \circ f_1(\w_t) - f_3(\u^2_t)} + \Norm{f_3(\u^2_t) - \u^3_t}\\
		& \leq L_f \Norm{\y_t^2 - \u^2_t}  + \Norm{f_3(\u^2_t) - \u^3_t}\\
		& \leq L_f(L_f \Norm{f_1(\w_t) - \u^1_t} + \Norm{f_2(\u_t^1) - \u_t^2}) +  \Norm{f_3(\u^2_t) - \u^3_t}\\
		& \cdots\\
		\Norm{\y_t^i-\u_t^i} & \leq L_f \Norm{\y_t^{i-1} - \u_t^{i-1}} + \Norm{f_i(\u_t^{i-1}) - \u_t^i}\\
		& \leq \sum_{j=1}^i L_f^{i-j} \Norm{f_j(\u_t^{j-1}) - \u_t^j}
	\end{align*}
	To this end, we can conclude that:
	\begin{align*}
		\Norm{\nabla F(\w_t) -\nabla \widehat{F}_K(\w_t)} & \leq \sum_{i=1}^{K-1} C_i \Norm{f_i(\u^{i-1}_t) - \u_t^i}, 
	\end{align*}
	where $C_i \coloneqq L_f^{K-1} L_J(1+L_f+\dotsc+L_f^{K-i-1})$. Since $C_i \leq L_F$, we know that:
\begin{align*}
\E\left[\Norm{\prod_{i=1}^K\nabla f_i(\u_t^{i-1}) - \nabla F(\x_t)}^2\right] \leq K\sum_{i=1}^{K-1} A_i^2\E\left[\Norm{f_i(\u_t^{i-1})-\u_t^i}^2\right].
\end{align*}
\end{proof}

Next, we bound the term $\E\left[\Norm{\v_t - \prod_{i=1}^K \nabla f_i(\u_t^{i-1}) }^2\right]$ and $\E\left[\Norm{\u_t^i - f_i(\u_t^{i-1})}^2\right]$ separately.


\begin{lemma}\label{lem:3} The gradient estimator $\v_t$ enjoys the following guarantee:
\begin{align*}
    \sum_{t=1}^T \E\left[\Norm{\v_t -\prod_{i=1}^K \nabla f_i(\u_t^{i-1})}^2\right] &\leq \frac{\E\left[\Norm{\v_1 -\prod_{i=1}^K \nabla f_i(\u_1^{i-1})}^2\right]}{\alpha }+ \frac{2 K^2 \alpha T\sigma_J^2}{B_1}  L_f^{2K-2}\\
    & \qquad +\frac{2 K}{\alpha B_1} \LL_J^2 L_f^{2K-2} \sum_{t=1}^T \sum_{i=1}^{K} \E\left[\Norm{\u_t^{i-1} - \u_{t-1}^{i-1}}^2\right]
\end{align*}
\end{lemma}	
\begin{proof}
According to the definition of $\v_t$, we know that:
\begin{align*}
\v_t = \frac{1}{B_1}\sum_{j=1}^{B_1} \prod_{i=1}^K \nabla f_i(\u_t^{i-1};\xi_t^{i,j})  + (1-\alpha_t)\left(\v_{t-1} - \frac{1}{B_1} \sum_{j=1}^{B_1} \prod_{i=1}^K \nabla f_i(\u_{t-1}^{i-1};\xi_t^{i,j})\right).    
\end{align*} 
Then we can obtain the following guarantee:
\begin{align*}
&\Norm{\v_t -\prod_{i=1}^K \nabla f_i(\u_t^{i-1})}^2\\
=&\left \| \frac{1}{B_1}\sum_{j=1}^{B_1}\prod_{i=1}^K \nabla f_i(\u_t^{i-1};\xi_t^{i,j})  + (1-\alpha_t)\left(\v_{t-1} - \frac{1}{B_1} \sum_{j=1}^{B_1} \prod_{i=1}^K \nabla f_i(\u_{t-1}^{i-1};\xi_t^{i,j})\right) - \prod_{i=1}^K \nabla f_i(\u_t^{i-1})\right \|^2 \\
=& \left \| (1-\alpha_t) \left( \v_{t-1} - \prod_{i=1}^K \nabla f_i(\u_{t-1}^{i-1}) \right) + \alpha_t \left( \frac{1}{B_1}\sum_{j=1}^{B_1}  \prod_{i=1}^K \nabla f_i(\u_{t-1}^{i-1};\xi_t^{i,j}) - \prod_{i=1}^K \nabla f_i(\u_{t-1}^{i-1})\right) \right.\\
& \left. + \left( \prod_{i=1}^K \nabla f_i(\u_{t-1}^{i-1}) -\prod_{i=1}^K \nabla f_i(\u_{t}^{i-1})   -  \frac{1}{B_1}\sum_{j=1}^{B_1}\prod_{i=1}^K \nabla f_i(\u_{t-1}^{i-1};\xi_t^{i,j}) + \frac{1}{B_1}\sum_{j=1}^{B_1}\prod_{i=1}^K \nabla f_i(\u_{t}^{i-1};\xi_t^{i,j})\right) \right\|^2 
\end{align*}

Since  the expectation of last two terms equals zero, such that 
\begin{gather*}
    E \left[  \frac{1}{B_1}\sum_{j=1}^{B_1} \prod_{i=1}^K \nabla f_i(\u_{t-1}^{i-1};\xi_t^{i,j}) - \prod_{i=1}^K \nabla f_i(\u_{t-1}^{i-1})\right] = 0, \\
\E\left[\prod_{i=1}^K \nabla f_i(\u_{t-1}^{i-1}) -\prod_{i=1}^K \nabla f_i(\u_{t}^{i-1}) -  \frac{1}{B_1}\sum_{j=1}^{B_1} \prod_{i=1}^K \nabla f_i(\u_{t-1}^{i-1};\xi_t^{i,j}) +  \frac{1}{B_1}\sum_{j=1}^{B_1} \prod_{i=1}^K \nabla f_i(\u_{t}^{i-1};\xi_t^{i,j}) \right] =0,
\end{gather*}
we would have that:
\begin{align*}
&\E\left[\Norm{\v_t -\prod_{i=1}^K \nabla f_i(\u_t^{i-1})}^2\right]\\
\leq & (1-\alpha)^2 \E\left[\Norm{ \v_{t-1} - \prod_{i=1}^K \nabla f_i(\u_{t-1}^{i-1})}^2\right] +  \E \left[ \left\| \alpha \left(\frac{1}{B_1}\sum_{j=1}^{B_1}  \prod_{i=1}^K \nabla f_i(\u_{t-1}^{i-1};\xi_t^{i,j}) - \prod_{i=1}^K \nabla f_i(\u_{t-1}^{i-1}) \right) \right.\right. \\
& \left.\left. +\left(\frac{1}{B_1}\sum_{j=1}^{B_1}  \prod_{i=1}^K \nabla f_i(\u_{t-1}^{i-1};\xi_t^{i,j}) -  \frac{1}{B_1}\sum_{j=1}^{B_1}  \prod_{i=1}^K \nabla f_i(\u_{t}^{i-1};\xi_t^{i,j}) -  \prod_{i=1}^K \nabla f_i(\u_{t-1}^{i-1}) + \prod_{i=1}^K \nabla f_i(\u_{t}^{i-1})\right) \right\|^2 \right] \\
\leq & (1-\alpha_t) \E\left[\Norm{ \v_{t-1} - \prod_{i=1}^K \nabla f_i(\u_{t-1}^{i-1})}^2\right] + 2\alpha_t^2\E\left[\Norm{\frac{1}{B_1}\sum_{j=1}^{B_1}  \prod_{i=1}^K \nabla f_i(\u_{t-1}^{i-1};\xi_t^{i,j}) - \prod_{i=1}^K \nabla f_i(\u_{t-1}^{i-1})}^2\right]\\
& +2\E\left[\Norm{\frac{1}{B_1}\sum_{j=1}^{B_1}  \prod_{i=1}^K \nabla f_i(\u_{t-1}^{i-1};\xi_t^{i,j}) -  \frac{1}{B_1}\sum_{j=1}^{B_1}  \prod_{i=1}^K \nabla f_i(\u_{t}^{i-1};\xi_t^{i,j}) -  \prod_{i=1}^K \nabla f_i(\u_{t-1}^{i-1}) + \prod_{i=1}^K \nabla f_i(\u_{t}^{i-1})}^2\right]
\end{align*}
Then, we would bound the last two terms, respectively. First, we have that:
\begin{align*}
    &\E\left[\Norm{ \frac{1}{B_1}\sum_{j=1}^{B_1}  \prod_{i=1}^K \nabla f_i(\u_{t-1}^{i-1};\xi_t^{i,j}) - \prod_{i=1}^K \nabla f_i(\u_{t-1}^{i-1})}^2\right]\\
    = & \frac{1}{B_1^2}\sum_{j=1}^{B_1}\E\left[\Norm{\prod_{i=1}^K \nabla f_i(\u_{t-1}^{i-1};\xi_t^{i,j}) - \prod_{i=1}^K \nabla f_i(\u_{t-1}^{i-1})}^2\right]\\
    = & \frac{1}{B_1}\E\left[\Norm{\prod_{i=1}^K \nabla f_i(\u_{t-1}^{i-1};\xi_t^{i,j}) - \prod_{i=1}^K \nabla f_i(\u_{t-1}^{i-1})}^2\right]\\
    =&\frac{1}{B_1} \E\left[\left\|\prod_{i=1}^K \nabla f_i(\u_{t-1}^{i-1};\xi_t^{i,j}) - \nabla f_1(\u_{t-1}^0) \prod_{i=2}^K \nabla  f_i(\u_{t-1}^{i-1};\xi_t^{i,j}) \right. \right. \\
    & \qquad\quad \left.\left. + \nabla f_1(\u_{t-1}^0) \prod_{i=2}^K \nabla  f_i(\u_{t-1}^{i-1};\xi_t^{i,j}) - \nabla f_1(\u_{t-1}^0) \nabla f_2(\u_{t-1}^1) \prod_{i=3}^K \nabla  f_i(\u_{t-1}^{i-1};\xi_t^{i,j}) \right. \right.\\ 
    &  \qquad\qquad \cdots \\
    & \qquad\quad \left.\left. +  \left(\prod_{i=1}^{K-1} \nabla f_i(\u_{t-1}^{i-1})\right) \nabla f_K(\u_{t-1}^{K-1};\xi_t^{i,j}) - \prod_{i=1}^K \nabla f_i(\u_{t-1}^{i-1}) \right\|^2\right] \\
    \leq & \frac{K}{B_1} \E\left[\Norm{\prod_{i=1}^K \nabla f_i(\u_{t-1}^{i-1};\xi_t^{i,j}) - \nabla f_1(\u_{t-1}^0) \prod_{i=2}^K \nabla  f_i(\u_{t-1}^{i-1};\xi_t^{i,j})}^2 \right] \\
    & + \frac{K}{B_1}\E\left[\Norm{\nabla f_1(\u_{t-1}^0) \prod_{i=2}^K \nabla  f_i(\u_{t-1}^{i-1};\xi_t^{i,j}) -\nabla f_1(\u_{t-1}^0) \nabla f_2(\u_{t-1}^1) \prod_{i=3}^K \nabla  f_i(\u_{t-1}^{i-1};\xi_t^{i,j})}^2 \right] \\
    & + ... \\
    & + \frac{K}{B_1}\E\left[\Norm{\left(\prod_{i=1}^{K-1} \nabla f_i(\u_{t-1}^{i-1})\right) \nabla f_K(\u_{t-1}^{K-1};\xi_t^{i,j})  - \prod_{i=1}^K \nabla f_i(\u_{t-1}^{i-1})}^2 \right]  \\
   \leq &  \frac{K}{B_1}\sigma_J^2(L_f^{2K-2} +  L_f^{2K-2}+...+L_f^{2K-2}) \\
   = & \frac{K^2 }{B_1} \sigma_J^2 L_f^{2K-2}
\end{align*}

When dealing with the second term, note that 
\begin{align*}
    \E\left[\left\|\frac{1}{B_1}\sum_{j=1}^{B_1} \left(  \prod_{i=1}^K \nabla f_i(\u_{t-1}^{i-1};\xi_t^{i,j}) -  \prod_{i=1}^K \nabla f_i(\u_{t}^{i-1};\xi_t^{i,j})   -  \prod_{i=1}^K \nabla f_i(\u_{t-1}^{i-1}) + \prod_{i=1}^K \nabla f_i(\u_{t}^{i-1})\right)\right\|\right]=0
\end{align*}
 So, we can have that:

\begin{align*}
    &\E\left[\Norm{\frac{1}{B_1}\sum_{j=1}^{B_1}  \prod_{i=1}^K \nabla f_i(\u_{t-1}^{i-1};\xi_t^{i,j}) -  \frac{1}{B_1}\sum_{j=1}^{B_1}  \prod_{i=1}^K \nabla f_i(\u_{t}^{i-1};\xi_t^{i,j}) -  \prod_{i=1}^K \nabla f_i(\u_{t-1}^{i-1}) + \prod_{i=1}^K \nabla f_i(\u_{t}^{i-1})}^2\right]\\
    = & \frac{1}{B_1^2} \E\left[\Norm{\sum_{j=1}^{B_1}\left(  \prod_{i=1}^K \nabla f_i(\u_{t-1}^{i-1};\xi_t^{i,j}) -  \prod_{i=1}^K \nabla f_i(\u_{t}^{i-1};\xi_t^{i,j}) -  \prod_{i=1}^K \nabla f_i(\u_{t-1}^{i-1}) + \prod_{i=1}^K \nabla f_i(\u_{t}^{i-1}) \right)}^2\right]\\
    = & \frac{1}{B_1^2}\sum_{j=1}^{B_1}\E\left[\Norm{  \prod_{i=1}^K \nabla f_i(\u_{t-1}^{i-1};\xi_t^{i,j}) -  \prod_{i=1}^K \nabla f_i(\u_{t}^{i-1};\xi_t^{i,j}) -  \prod_{i=1}^K \nabla f_i(\u_{t-1}^{i-1}) + \prod_{i=1}^K \nabla f_i(\u_{t}^{i-1}) }^2\right]\\
    = & \frac{1}{B_1}\E\left[\Norm{  \prod_{i=1}^K \nabla f_i(\u_{t-1}^{i-1};\xi_t^{i,j}) -  \prod_{i=1}^K \nabla f_i(\u_{t}^{i-1};\xi_t^{i,j}) -  \prod_{i=1}^K \nabla f_i(\u_{t-1}^{i-1}) + \prod_{i=1}^K \nabla f_i(\u_{t}^{i-1}) }^2\right]\\
    \leq & \frac{1}{B_1}\E\left[\Norm{  \prod_{i=1}^K \nabla f_i(\u_{t-1}^{i-1};\xi_t^{i,j}) -  \prod_{i=1}^K \nabla f_i(\u_{t}^{i-1};\xi_t^{i,j}) }^2\right]\\
    =&\frac{1}{B_1} \E\left[ \left\|\prod_{i=1}^K \nabla f_i(\u_{t-1}^{i-1};\xi_t^{i,j}) - \nabla f_1(\u_{t}^{0};\xi_t^{1,j}) \prod_{i=2}^K \nabla f_i(\u_{t-1}^{i-1};\xi_t^{i,j})   \right. \right. \\
    & \left. \left.\quad\quad + \nabla f_1(\u_{t}^{0};\xi_t^{1,j}) \prod_{i=2}^K \nabla f_i(\u_{t-1}^{i-1};\xi_t^{i,j})  - \nabla f_1(\u_{t}^{0};\xi_t^{1,j}) \nabla f_2(\u_{t}^1;\xi_t^{2,j}) \prod_{i=3}^K \nabla f_i(\u_{t-1}^{i-1};\xi_t^{i,j})\right. \right.  \\
    &\left. \left.\quad\quad \cdots \right. \right. \\
    & \left. \left.\quad\quad+ \left(\prod_{i=1}^{K-1} \nabla f_i(\u_{t}^{i-1};\xi_t^{i,j})\right)\nabla f_K(\u_{t-1}^{K-1};\xi_t^{K,j}) - \prod_{i=1}^K \nabla f_i(\u_{t}^{i-1};\xi_t^{i,j})  \right\|^2 \right]\\
    \leq & \frac{K}{B_1} \E\left[\Norm{\prod_{i=1}^K \nabla f_i(\u_{t-1}^{i-1};\xi_t^{i,j}) - \nabla f_1(\u_{t}^{0};\xi_t^{1,j}) \prod_{i=2}^K \nabla f_i(\u_{t-1}^{i-1};\xi_t^{i,j})}^2 \right] \\
    &\quad + \frac{K}{B_1}\E\left[\Norm{\nabla f_1(\u_{t}^{0};\xi_t^{1,j}) \prod_{i=2}^K \nabla f_i(\u_{t-1}^{i-1};\xi_t^{i,j}) - \nabla f_1(\u_{t}^{0};\xi_t^{1,j}) \nabla f_2(\u_{t}^1;\xi_t^{2,j}) \prod_{i=3}^K \nabla f_i(\u_{t-1}^{i-1};\xi_t^{i,j})}^2 \right] \\
    & \quad \cdots  \\
    & \quad +  \frac{K}{B_1}\E\left[\Norm{\left(\prod_{i=1}^{K-1} \nabla f_i(\u_{t}^{i-1};\xi_t^{i,j})\right)\nabla f_K(\u_{t-1}^{K-1};\xi_t^{i,j})- \prod_{i=1}^K \nabla f_i(\u_{t}^{i-1};\xi_t^{i,j})}^2 \right]\\
    \leq & \frac{K}{B_1} \LL_J^2L_f^{2K-2} \sum_{i=1}^{K} \E\left[\Norm{\u_t^{i-1} - \u_{t-1}^{i-1}}^2\right]
\end{align*}

To this end, we can conclude that:
\begin{align*}
&\E\left[\Norm{\v_t -\prod_{i=1}^K \nabla f_i(\u_t^{i-1})}^2\right] \\
\leq & (1-\alpha_t) \E\left[\Norm{ \v_{t-1} - \prod_{i=1}^K \nabla f_i(\u_{t-1}^{i-1})}^2\right]+ \frac{2K^2\alpha_t^2\sigma_J^2}{B_1} L_f^{2K-2} +\frac{2K}{B_1}  \LL_J^2L_f^{2K-2} \sum_{i=1}^{K} \E\left[\Norm{\u_t^{i-1} - \u_{t-1}^{i-1}}^2\right],
\end{align*}
which implies that 
\begin{align*}
    &\sum_{t=1}^T \E\left[\Norm{\v_t -\prod_{i=1}^K \nabla f_i(\u_t^{i-1})}^2\right] \leq \sum_{t=2}^T\left(\frac{1}{\alpha_{t+1}}-\frac{1}{\alpha_t} \right)\E\left[\Norm{\v_1 -\prod_{i=1}^K \nabla f_i(\u_1^{i-1})}^2\right]\\
    & +\frac{1}{\alpha_2}\E\left[\Norm{\v_1 -\prod_{i=1}^K \nabla f_i(\u_1^{i-1})}^2\right]+ \frac{2 \sigma_J^2 K^2  L_f^{2K-2}}{B_1}   \sum_{t=1}^T\alpha_{t+1} +\frac{2 K}{B_1}  \LL_J^2L_f^{2K-2} \sum_{t=1}^T \frac{1}{\alpha_{t+1}}\sum_{i=1}^{K} \E\left[\Norm{\u_t^{i-1} - \u_{t-1}^{i-1}}^2\right]
\end{align*}
Note that  $1 / \alpha_{t+1}-1 / \alpha_{t}=t^{1 / 2}-(t-1)^{1 / 2} \leq 1 / 2$ and for $t \geq 2$, $\sum_{t=1}^{T} t^{-2 / 3} \leqslant 3 T^{1 / 3}+1 $.
As a result, we have that 
\begin{align*}
    &\sum_{t=1}^T \E\left[\Norm{\v_t -\prod_{i=1}^K \nabla f_i(\u_t^{i-1})}^2\right] \leq \\
    & 4\E\left[\Norm{\v_1 -\prod_{i=1}^K \nabla f_i(\u_1^{i-1})}^2\right]+ \frac{16 \sigma_J^2 K^2  L_f^{2K-2} }{B_1}T^{1/3}    +\frac{4 K}{B_1}  \LL_J^2L_f^{2K-2} \sum_{t=1}^T \frac{1}{\alpha_{t+1}}\sum_{i=1}^{K} \E\left[\Norm{\u_t^{i-1} - \u_{t-1}^{i-1}}^2\right]
\end{align*}
\end{proof}
\begin{lemma}\label{lem:4} The inner function estimator $\u_t$ ensures that:
\begin{align*}
	\sum_{t=1}^T \sum_{i=1}^K \E\left[\Norm{\u_t^i -  f_i(\u_t^{i-1})}^2\right] \leq &\frac{\sum_{i=1}^K \E\left[\Norm{\u_1^i -  f_i(\u_1^{i-1})}^2\right]}{\alpha} + \frac{2 \alpha K T  \sigma^2}{B_1}  + \frac{2\LL_f^2}{\alpha B_1}  \sum_{t=1}^T \sum_{i=1}^K\E\left[\Norm{\u_{t}^{i-1} - \u_{t-1}^{i-1}}^2\right]
\end{align*}
\end{lemma}	
\begin{proof}
Since $\u_t^i = \frac{1}{B_1}\sum_{j=1}^{B_1}f_i(\u_t^{i-1};\xi_t^{i,j}) + (1-\alpha)(\u_{t-1}^i - \frac{1}{B_1}\sum_{j=1}^{B_1}f_i(\u_{t-1}^{i-1};\xi_t^{i,j}))$, we have:
\begin{align*}
& \E\left[\Norm{f_i(\u_t^{i-1}) - \u_t^i}^2\right]\\
=& \E\left[\left\|(1-\alpha)\left(\u_{t-1}^i - f_i(\u_{t-1}^{i-1})\right) +  \frac{\alpha}{B_1}\sum_{j=1}^{B_1} \left(  f_i(\u_{t-1}^{i-1};\xi_t^{i,j}) -  f_i(\u_{t-1}^{i-1})\right) \right.\right.\\
&\quad\quad\quad \left.\left.   +  f_i(\u_{t-1}^{i-1}) - f_i(\u_t^{i-1})  -\frac{1}{B_1}\sum_{j=1}^{B_1} \left(f_i(\u_{t-1}^{i-1};\xi_t^{i,j}) -  f_i(\u_t^{i-1};\xi_t^{i,j}) \right) \right\|^2\right] \\
\leq & (1-\alpha)^2 \E \Norm{\u_{t-1}^i - f_i(\u_{t-1}^{i-1})}^2 + \frac{2\alpha^2}{B_1} \E\left[\Norm{f_i(\u_{t-1}^{i-1};\xi_t^{i,j}) - f_i(\u_{t-1}^{i-1})}^2\right] \\
& \qquad  +2\E\left[\left\|\frac{1}{B_1}\sum_{j=1}^{B_1}\left(f_i(\u_{t-1}^{i-1};\xi_t^{i,j}) - f_i(\u_t^{i-1};\xi_t^{i,j} )- f_i(\u_{t-1}^{i-1}) + f_i(\u_t^{i-1})\right)\right\|^2\right]
\end{align*}
where the last inequality is due to: $\E\left[f_i(\u_{t-1}^{i-1})   - f_i(\u_t^{i-1})  -  \frac{1}{B_1}\sum_{j=1}^{B_1}\left( f_i(\u_{t-1}^{i-1};\xi_t^{i,j}) - f_i(\u_t^{i-1};\xi_t^{i,j})\right)\right] = 0 $, as well as $\E\left[\frac{\alpha}{B_1}\sum_{j=1}^{B_1} \left(  f_i(\u_{t-1}^{i-1};\xi_t^{i,j}) -  f_i(\u_{t-1}^{i-1})\right)\right]=0$. 

Also, since we know the fact that $\E\left[f_i(\u_{t-1}^{i-1})   - f_i(\u_t^{i-1})  -  f_i(\u_{t-1}^{i-1};\xi_t^{i,j}) - f_i(\u_t^{i-1};\xi_t^{i,j})\right] = 0 $, we have:
\begin{align*}
&\E\left[\left\|\frac{1}{B_1}\sum_{j=1}^{B_1}\left(f_i(\u_{t-1}^{i-1};\xi_t^{i,j})  - f_i(\u_t^{i-1};\xi_t^{i,j} )- f_i(\u_{t-1}^{i-1}) + f_i(\u_t^{i-1})\right)\right\|^2\right]\\
\leq &  \frac{1}{B_1^2}\sum_{j=1}^{B_1}\E\left[\left\|\left(f_i(\u_{t-1}^{i-1};\xi_t^{i,j})  - f_i(\u_{t-1}^{i-1}) + f_i(\u_t^{i-1}) - f_i(\u_t^{i-1};\xi_t^{i,j})\right)\right\|^2\right]\\
\leq & \frac{1}{B_1^2}\sum_{j=1}^{B_1}\E\left[\left\|\left(f_i(\u_{t-1}^{i-1};\xi_t^{i,j})  -f_i(\u_t^{i-1};\xi_t^{i,j})\right)\right\|^2\right]\\
\leq &  \frac{1}{B_1}\E\left[\left\|f_i(\u_{t-1}^{i-1};\xi_t^{i,j})  -f_i(\u_t^{i-1};\xi_t^{i,j})\right\|^2\right]\\
\leq &  \frac{1}{B_1}\LL_f^2 \left\|\u_{t-1}^{i-1} - \u_t^{i-1}\right\|^2.
\end{align*}
As a result, we can conclude that:
\begin{align*}
\E\left[\Norm{f_i(\u_t^{i-1}) - \u_t^i}^2\right] \leq&  (1-\alpha) \E \Norm{\u_{t-1}^i - f_i(\u_{t-1}^{i-1})}^2 + \frac{2\alpha^2 \sigma^2}{B_1} + \frac{2}{B_1}\LL_f^2 \left\|\u_{t-1}^{i-1} - \u_t^{i-1}\right\|^2  
\end{align*}
This leads to the fact that:
\begin{align*}
\sum_{i=1}^K \E\left[\Norm{f_i(\u_t^{i-1}) - \u_t^i}^2\right] \leq&  (1-\alpha) \sum_{i=1}^K \E \Norm{\u_{t-1}^i - f_i(\u_{t-1}^{i-1})}^2 + \frac{2\alpha^2 \sigma^2 K}{B_1} + \frac{2}{B_1}\LL_f^2 \sum_{i=1}^K \left\|\u_{t-1}^{i-1} - \u_t^{i-1}\right\|^2  
\end{align*}
By summing up and rearranging, we can get:
\begin{align*}
	\sum_{t=1}^T \sum_{i=1}^K \E\left[\Norm{\u_t^i -  f_i(\u_t^{i-1})}^2\right] \leq &\frac{\sum_{i=1}^K \E\left[\Norm{\u_1^i -  f_i(\u_1^{i-1})}^2\right]}{\alpha} + \frac{2 \alpha K T  \sigma^2}{B_1} + \frac{2\LL_f^2}{\alpha B_1}  \sum_{t=1}^T \sum_{i=1}^K\E\left[\Norm{\u_{t}^{i-1} - \u_{t-1}^{i-1}}^2\right],
\end{align*}
which finishes the proof.
\end{proof}


Then, we try to bound the term $\sum_{i=1}^K\E\left[\Norm{\u_{t+1}^{i-1} - \u_{t}^{i-1}}^2\right]$.
\begin{lemma}
\begin{align*}
\sum_{i=1}^K  \E\left[\Norm{\u_{t+1}^{i-1} - \u_{t}^{i-1}}^2 \right]  
 \leq&  \left(\sum_{i=1}^{K}\left(2\LL_f^2\right)^{i-1}\right) \left(\E \left[\eta^2 \Norm{\z_t - \x_t}^2\right] + \frac{2\alpha^2\sigma^2 K}{B_1}   + 2\alpha^2K \sum_{i=1}^{K} \E\left[\Norm{\u_t^i - f_i(\u_t^{i-1})}^2\right]  \right) 
\end{align*}
\end{lemma}
\begin{proof}
We discuss the following two cases, separately.
\begin{itemize}
	\item[1.] For the first level, i.e., $i=1$, we have: 
    \begin{align*}
	    \E\left[\Norm{\u_{t+1}^{i-1} - \u_{t}^{i-1}}^2\right] = \E\left[\Norm{\x_{t+1} - \x_{t}}^2\right] = \E \left[\eta^2 \Norm{\z_t - \x_t}^2\right].
	\end{align*}
	\item[2.] For other levels, i.e., $2\leq i\leq K$, we have:
	\begin{align*}
		& \E\left[\Norm{\u_{t+1}^{i-1} - \u_{t}^{i-1}}^2\right]\\
		=&\E\left[\left\|\alpha \left(f_{i-1}(\u_{t}^{i-2}) - \u_{t}^{i-1}\right) + \frac{1}{B_1}\sum_{j=1}^{B_1}(f_{i-1}(\u_{t+1}^{i-2};\xi_{t+1}^{i-1}) - f_{i-1}(\u_{t}^{i-2};\xi_{t+1}^{i-1}))\right.\right.\\
		 &\quad \left.\left.+ \alpha \left( \frac{1}{B_1}\sum_{j=1}^{B_1}f_{i-1} (\u_{t}^{i-2};\xi_{t+1}^{i-1})-f_{i-1}(\u_{t}^{i-2})\right) \right\|^2 \right]\\
\leq& 2\E\left[\Norm{\alpha\left(f_{i-1}(\u_{t}^{i-2}) - \u_{t}^{i-1}\right) + \alpha \left(\frac{1}{B_1}\sum_{j=1}^{B_1}f_{i-1} (\u_{t}^{i-2};\xi_{t+1}^{i-1})-f_{i-1}(\u_{t}^{i-2})\right)}^2\right] \\
& \qquad + 2\LL_f^2\E\left[\Norm{\u_{t+1}^{i-2} - \u_{t}^{i-2}}^2 \right]\\
	\leq&  2\alpha^2 \Norm{f_{i-1}(\u_{t}^{i-2}) - \u_{t}^{i-1}}^2 + \frac{2\alpha^2 \sigma^2}{B_1} + 2\LL_f^2\E\left[\Norm{\u_{t+1}^{i-2} - \u_{t}^{i-2}}^2 \right].
	\end{align*} 
\end{itemize}

Denote $\Upsilon_t^i =\E\left[\Norm{\u_t^i - f_i(\u_t^{i-1})}^2\right]$ and  $Q^{i} = \E\left[\Norm{\u_{t+1}^{i-1} - \u_{t}^{i-1}}^2\right]$, we have $Q^{i} \leq 2\LL_f^2 Q^{i-1} + 2\alpha^2 \Upsilon_t^{i-1} + \frac{2\alpha^2\sigma^2}{B_1}$ for $i\geq 2$. Then we can get:
\begin{align*}
Q^{1} &\leq \E \left[\eta^2 \Norm{\z_t - \x_t}^2\right] \\
Q^{2} &\leq \left(2\LL_f^2\right) \E \left[\eta^2 \Norm{\z_t - \x_t}^2\right] +  \frac{2\alpha^2\sigma^2}{B_1} +2\alpha^2 \Upsilon_t^{1} \\
Q^{3} &\leq \left(2\LL_f^2\right)^2 \E \left[\eta^2 \Norm{\z_t - \x_t}^2\right]  +  \frac{2\alpha^2\sigma^2 \left(1+2\LL_f^2\right)}{B_1} + 2\alpha^2\left( 2\LL_f^2  \Upsilon_t^{1}+ \Upsilon_t^{2}\right) \\
& \cdots\\
Q^{i} &\leq \left(2\LL_f^2\right)^{i-1} \E \left[\eta^2 \Norm{\z_t - \x_t}^2\right] + \frac{2\alpha ^2\sigma^2}{B_1} \sum_{j=1}^{i-1}\left(2\LL_f^2\right)^{j-1}  + 2\alpha^2 \sum_{j=1}^{i-1} \left(2\LL_f^2\right)^{i-1-j} \Upsilon_t^{j}\\
&\leq \left(2\LL_f^2\right)^{i-1} \E \left[\eta^2 \Norm{\z_t - \x_t}^2\right] + \frac{2\alpha^2\sigma^2}{B_1} \sum_{j=1}^{K}\left(2\LL_f^2\right)^{j-1}  + 2\alpha^2 \sum_{j=1}^{K} \sum_{l=1}^{K} \left(2L_f^2\right)^{K-l} \Upsilon_t^{j}.
\end{align*}
When summing up, we have:
\begin{align*}
    \sum_{i=1}^{K} Q^{i} &\leq \sum_{i=1}^{K}\left(2\LL_f^2\right)^{i-1} \E \left[\eta^2 \Norm{\z_t - \x_t}^2\right] + \frac{2\alpha^2\sigma^2 K}{B_1} \sum_{i=1}^{K}\left(2\LL_f^2\right)^{i-1} + 2\alpha^2 K \sum_{j=1}^{K} \sum_{l=1}^{K} \left(2L_f^2\right)^{K-l} \Upsilon_t^{j}  \\
    &\leq \left(\sum_{i=1}^{K}\left(2\LL_f^2\right)^{i-1}\right) \left(\E \left[\eta^2 \Norm{\z_t - \x_t}^2\right] + \frac{2\alpha^2\sigma^2 K}{B_1}   + 2\alpha^2K \sum_{i=1}^{K}  \Upsilon_t^{i}\right)
\end{align*}
So we have :
\begin{align*}
  \sum_{i=1}^K \E\left[\Norm{\u_{t+1}^{i-1} - \u_{t}^{i-1}}^2 \right]  
 \leq&  \left(\sum_{i=1}^{K}\left(2\LL_f^2\right)^{i-1}\right) \left(\E \left[\eta^2 \Norm{\z_t - \x_t}^2\right] + \frac{2\alpha^2\sigma^2 K}{B_1}   + 2\alpha^2K \sum_{i=1}^{K} \E\left[\Norm{\u_t^i - f_i(\u_t^{i-1})}^2\right]  \right) 
\end{align*}
\end{proof}
\begin{lemma} \label{lem:66}
Denote that the constants as $L_1 = \left(2K \LL_J^2 L_f^{2K-2} + 4\LL_f^2\right)\left(\sum_{i=1}^{K}\left(2\LL_f^2\right)^{i-1}\right)$, $L_2 = \max\{ 2,2KL_F^2\}$, $L_3 =  2L_2(2K\sigma_J^2L_f^{2K-2} + 4K \sigma^2 + 2L_1 \sigma^2 K + L_1 D^2 + L_1) $, we can ensure that:
\begin{align*}
	\sum_{t=1}^T \E\left[\Norm{\nabla F(\x_t) - \v_t }^2\right] \leq \frac{L_3}{\alpha B_0} + \frac{\alpha L_3 T}{B_1} + \frac{L_3\eta^2 T}{\alpha B_1}.
\end{align*}
\end{lemma}

\begin{proof}

Based on Lemma \ref{lem:2}, we have:
\begin{align*}
\sum_{t=1}^T \E\left[\Norm{\nabla F(\x_t) - \v_t }^2\right] \leq & 2 \sum_{t=1}^T\E\left[\Norm{\v_t - \prod_{i=1}^K \nabla f_i(\u_t^{i-1}) }^2\right]\quad + 2KL_F^2 \sum_{t=1}^T \sum_{i=1}^{K-1} \E\left[\Norm{\u_t^i - f_i(\u_t^{i-1})}^2\right].
\end{align*}
Noting that $\E\left[\Norm{\v_1 -\prod_{i=1}^K \nabla f_i(\u_1^{i-1})}^2\right] \leq \frac{K^2 \sigma_J^2 L_f^{2K-2}}{B_0}$, $\sum_{i=1}^K \E\left[\Norm{\u_t^i -  f_i(\u_t^{i-1})}^2\right] \leq \frac{K \sigma^2}{B_0}$ and setting $2 \alpha L_1 K \leq B_1$, we can deduce that:
\begin{align*}
 & \sum_{t=1}^T\E\left[\Norm{\v_t - \prod_{i=1}^K \nabla f_i(\u_t^{i-1}) }^2\right] + 2\sum_{t=1}^T \sum_{i=1}^{K-1} \E\left[\Norm{\u_t^i - f_i(\u_t^{i-1})}^2\right]\\
 \leq & \frac{K^2 \sigma_J^2 L_f^{2K-2} + 2K \sigma^2}{\alpha B_0} + \frac{2\alpha K^2 T\sigma_J^2L_f^{2K-2} + 4 \alpha K T  \sigma^2}{B_1} +\frac{2K\LL_J^2L_f^{2K-2}+4\LL_f^2}{\alpha B_1}\sum_{t=1}^T \sum_{i=1}^{K} \E\left[\Norm{\u_t^{i-1} - \u_{t-1}^{i-1}}^2\right]\\
\leq & \frac{K^2 \sigma_J^2 L_f^{2K-2} + 2K \sigma^2}{\alpha B_0} + \frac{2\alpha K^2 T\sigma_J^2L_f^{2K-2} + 4 \alpha K T  \sigma^2}{B_1} \\
 & +\frac{L_1}{\alpha B_1} \left( T \eta^2 D^2 + \frac{2T\alpha^2\sigma^2 K}{B_1}   + 2\alpha^2K \sum_{t=1}^T \sum_{i=1}^{K} \E\left[\Norm{\u_t^i - f_i(\u_t^{i-1})}^2\right]  \right) \\
 \leq & \frac{K^2 \sigma_J^2 L_f^{2K-2} + 2K \sigma^2}{\alpha B_0} + \frac{2\alpha K^2 T\sigma_J^2L_f^{2K-2} + 4 \alpha K T  \sigma^2}{B_1}  + \frac{L_1\eta^2D^2 T}{\alpha B_1} + \frac{2\alpha L_1 \sigma^2 KT}{B_1^2}\\
 & + \sum_{t=1}^T \sum_{i=1}^{K}\E\left[\Norm{\u_t^i - f_i(\u_t^{i-1})}^2\right] 
\end{align*}

So, we have that:
\begin{align*}
 & \sum_{t=1}^T \E\left[\Norm{\nabla F(\x_t) - \v_t }^2\right] \\
 \leq & L_2 \sum_{t=1}^T  \E\left[\Norm{\v_t - \prod_{i=1}^K \nabla f_i(\u_t^{i-1}) }^2\right] + L_2 \sum_{t=1}^T \sum_{i=1}^{K-1} \E\left[\Norm{\u_t^i - f_i(\u_t^{i-1})}^2\right]\\
\leq & L_2 \left(\frac{K^2 \sigma_J^2 L_f^{2K-2} + 2K \sigma^2}{\alpha B_0} + \frac{2\alpha K^2 T\sigma_J^2L_f^{2K-2} + 4 \alpha K T  \sigma^2}{B_1}  + \frac{L_1\eta^2D^2 T}{\alpha B_1} + \frac{2\alpha L_1 \sigma^2 KT}{B_1^2} \right) \\
\leq &  \frac{L_3}{\alpha B_0} + \frac{\alpha L_3 T}{B_1} + \frac{L_3\eta^2 T}{\alpha B_1}
\end{align*}
\end{proof}	

Now we can finish the proof as follows. Denote the Frank-Wolfe Gap as $\F(\x):=\max _{\hat{\x} \in \X}\langle \hat{\x}-\x,-\nabla F(\x)\rangle$ and $\z_t^{\star} = \arg \max _{\hat{\x} \in \X}\langle \hat{\x}-\x,-\nabla F(\x)\rangle$.

\begin{align*} 
F\left(\x_{t+1}\right) 
& \leq F\left(\x_t\right)+\left\langle\nabla F\left(\x_t\right), \x_{t+1}-\x_t\right\rangle+\frac{L_F}{2}\left\|\x_{t+1}-\x_{t}\right\|^{2} \\ 
& =F\left(\x_{t}\right)+\eta\left\langle \v_{t}, \z_{t}-\x_{t}\right\rangle+\eta\left\langle\nabla F\left(\x_{t}\right)-\v_t, \z_t-\x_t\right\rangle+\eta^{2} \frac{L_F}{2} D^{2} \\ 
& \leq F\left(\x_{t}\right)+\eta\left\langle \v_t, \z^{\star}_{t}-\x_{t}\right\rangle+\eta\left\langle\nabla F\left(\x_{t}\right)-\v_{t}, \z_{t}-\x_{t}\right\rangle+\eta^{2} \frac{L_F}{2} D^{2} \\ 
& =F\left(\x_t\right)+\eta\left\langle\nabla F\left(\x_t\right), \z^{\star}_{t}-\x_{t}\right\rangle+\eta\left\langle\nabla F\left(\x_{t}\right)-\v_{t}, \z_t-\z^{\star}_{t}\right\rangle+\eta^{2} \frac{L_F}{2} D^{2} \\ 
& \leq F\left(\x_{t}\right)-\eta \F\left(\x_{t}\right)+\eta D\left\|\nabla F\left(\x_t\right)-\v_t\right\|+\eta^{2} \frac{L_F}{2} D^{2}
\end{align*}
That is to say:
\begin{align*} 
\quad \frac{1}{T}\sum_{t=1}^{T}  \F\left(\x_{t}\right)& \leq \frac{F\left(\x_{1}\right)-F(\x_{T+1})}{\eta T}+\frac{D}{T} \sum_{t=1}^T \left\|\nabla F\left(\x_t\right)-\v_t\right\|+\eta \frac{L_F}{2} D^{2} 
\end{align*} 

By setting $T = \mathcal{O} \left(\epsilon^{-2}\right)$, $\eta =\mathcal{O} \left(\epsilon\right), \alpha = \mathcal{O} \left(\epsilon\right)$, $B_0 = B_1 = \mathcal{O} \left(\epsilon^{-1}\right)$, we can ensure that $\F(\x) \leq \epsilon$. Moreover, we can obtain the same guarantee by setting that $\alpha =\mathcal{O}\left( \epsilon^2\right), \eta = \mathcal{O}\left( \epsilon^2\right), B_1 =  \mathcal{O}\left(1 \right), B_0 = \mathcal{O} \left(\epsilon^{-1}\right), T = \mathcal{O}\left( \epsilon^{-3} \right)$.

\section{Proof of Gradient Mapping~(Theorem~\ref{thm:2} and Theorem~\ref{thm:2_0})}
In the previous analysis of Lemma~\ref{lem:66}, we simply reduce ${\eta^2}\Norm{\z_t - \x_t}^{2} \leq \eta^2 D^2$. To obtain the optimal rate, we have to keep this term. That is to say, we rewrite the Lemma~\ref{lem:66} as follows
\begin{align*}
			\frac{1}{T} \sum_{t=1}^{T} \E\left[\Norm{\nabla F(\x_t) - \v_t }^2\right] \leq \frac{L_3}{\alpha B_0 T}+ \frac{\alpha L_3}{B_1}+  \frac{\eta^2 L_3} {\alpha B_1T} \sum_{t=1}^{T}\Norm{\z_t - \x_t}^{2}.
\end{align*}

According to Proposition 2 of~\citet{NEURIPS2022_7e16384b}, we know that the gradient mapping 
\begin{align*}
  \left\|\mathcal{G}(\x_t, \beta)\right\|^{2} \leq -4 \beta g(\x_t, \v_t)+2\E\left[\Norm{\nabla F(\x_t) - \v_t }^2\right],  
\end{align*}  
where $g(\x_t, \v_t)=\min _{\y \in \X}\left\{\langle \v_t, \y-\x \rangle+\frac{\beta}{2}\|\y-\x_t\|^{2}\right\}$.

Due to the convergence of Frank-Wolfe algorithm~\cite{pmlr-v28-jaggi13}, we know that 
\begin{align*}
    \langle \v_t, \z_t-\x \rangle+\frac{\beta}{2}\|\z_t-\x_t\|^{2} \leq g(\x_t, \v_t) + \frac{2 \beta D^2}{N+2},
\end{align*}
which is widely used in the analysis of Frank-Wolfe algorithm~\cite{NEURIPS2022_7e16384b, Zhang2019OneSS,AAAI:2021:Wan:B,AAAI:2021:Wan:C}.
Now, we begin our proof
\begin{align*} 
&F\left(\x_{t+1}\right) \\
& \leq F\left(\x_t\right)+\left\langle\nabla F\left(\x_t\right), \x_{t+1}-\x_t\right\rangle+\frac{L_F}{2}\left\|\x_{t+1}-\x_{t}\right\|^{2} \\ 
& \leq F\left(\x_{t}\right)+\eta\left\langle\nabla F\left(\x_{t}\right), \z_{t}-\x_{t}\right\rangle+\eta^{2} \frac{L_F}{2} \Norm{\z_t - \x_t}^{2} \\ 
& =F\left(\x_{t}\right)+\eta\left\langle \v_{t}, \z_{t}-\x_{t}\right\rangle+\eta\left\langle\nabla F\left(\x_{t}\right)-\v_t, \z_t-\x_t\right\rangle+\eta^{2} \frac{L_F}{2} \Norm{\z_t - \x_t}^{2} \\ 
& =F\left(\x_{t}\right)+\eta\left\langle \v_{t}, \z_{t}-\x_{t}\right\rangle+\frac{\eta\beta}{2}\|\z_t-\x_t\|^{2}+\eta\left\langle\nabla F\left(\x_{t}\right)-\v_t, \z_t-\x_t\right\rangle - \frac{\eta\beta}{2}\|\z_t-\x_t\|^{2} +\eta^{2} \frac{L_F}{2} \Norm{\z_t - \x_t}^{2}
\end{align*}    
Denote $\y^{\star}= \min _{\y \in \X}\left\{\langle \v_t, \y-\x \rangle+\frac{\beta}{2}\|\y-\x_t\|^{2}\right\}$. Set  $\eta \leq \frac{\beta}{2 L_F}$~(note that $\beta$ is a positive constant), and then we have, 
\begin{align*} 
F\left(\x_{t+1}\right) 
& \leq F\left(\x_{t}\right)+\eta\left\langle \v_{t}, \y^{\star}-\x_{t}\right\rangle+\frac{\eta\beta}{2}\|\y^{\star}-\x_t\|^{2}+\frac{2 \eta \beta D^2}{N+2}\\
&\quad +\eta\left\langle\nabla F\left(\x_{t}\right)-\v_t, \z_t-\x_t\right\rangle - \frac{\eta\beta}{2}\|\z_t-\x_t\|^{2} +\eta^{2} \frac{L_F}{2} \Norm{\z_t - \x_t}^{2} \\
& \leq F(\x_t) + \eta g(\x_t, \v_t) + \frac{2\beta D^2  \eta}{N+2} +\eta\left\langle\nabla F\left(\x_{t}\right)-\v_t, \z_t-\x_t\right\rangle - \frac{\eta\beta}{2}\|\z_t-\x_t\|^{2} \\
& \qquad + \eta^{2} \frac{L_F}{2} \Norm{\z_t - \x_t}^{2} \\
& \leq F(\x_t) + \eta g(\x_t, \v_t) + \frac{2\beta D^2 \eta}{N+2} +\eta\left\langle\nabla F\left(\x_{t}\right)-\v_t, \z_t-\x_t\right\rangle - \frac{\eta\beta}{4}\|\z_t-\x_t\|^{2} \\
& \leq F(\x_t) + \eta g(\x_t, \v_t) + \frac{2\beta D^2  \eta}{N+2} +\frac{2\eta}{\beta} \E\left[\Norm{\nabla F(\x_t) - \v_t }^2\right] - \frac{\eta\beta}{8}\|\z_t-\x_t\|^{2}
\end{align*}    
As a result,
\begin{align*} 
-g(\x_t, \v_t)
& \leq \frac{F(\x_t) - F\left(\x_{t+1}\right)}{\eta}  + \frac{2\beta D^2 }{N+2} +\frac{2}{\beta} \E\left[\Norm{\nabla F(\x_t) - \v_t }^2\right] - \frac{\beta}{8}\|\z_t-\x_t\|^{2}
\end{align*}
So, we have:
\begin{align*}
  &\left\|\mathcal{G}(\x_t, \beta)\right\|^{2} \\
  \leq &-4 \beta g(\x_t, \v_t)+2\E\left[\Norm{\nabla F(\x_t) - \v_t }^2\right]  \\
  \leq &\frac{4\beta (F(\x_t) - F\left(\x_{t+1}\right))}{\eta}  + \frac{8\beta^2 D^2 }{N+2} + 10 \E\left[\Norm{\nabla F(\x_t) - \v_t }^2\right] - \frac{\beta^2}{2}\|\z_t-\x_t\|^{2}
\end{align*}  
Finally, we have
\begin{align*}
  &\quad \frac{1}{T}\sum_{t=1}^{T}\left\|\mathcal{G}(\x_t, \beta)\right\|^{2}   \\
  & \leq \frac{4\beta \Delta_F}{\eta T}  + \frac{8\beta^2 D^2 }{N} + \frac{10}{T}\sum_{t=1}^{T}\E\left[\Norm{\nabla F(\x_t) - \v_t }^2\right] - \frac{\beta^2}{2} \frac{1}{T}\sum_{t=1}^{T}\|\z_t-\x_t\|^{2} \\
  & \leq \frac{4\beta \Delta_F}{\eta T}  + \frac{8\beta^2 D^2 }{N} + \frac{10L_3}{ \alpha B_0 T}+ \frac{10 L_3\alpha}{B_1} + \left( 10 L_3 \frac{\eta^2} {\alpha B_1} -\frac{\beta^2}{2}\right)\frac{1}{T}  \sum_{t=1}^{T}\Norm{\z_t - \x_t}^{2} \\
  & \leq \frac{4\beta \Delta_F}{\eta T}  + \frac{8\beta^2 D^2 }{N} + \frac{10L_3}{ \alpha B_0 T}+ \frac{10 L_3\alpha}{B_1}
\end{align*}  
The last inequality holds with $\eta \leq \sqrt{\frac{\beta^2 \alpha B_1}{20 L_3}}$.
By setting $\alpha = \mathcal{O}(\sqrt{\epsilon})$, $\eta = \mathcal{O}(1)$, $T=\mathcal{O}(\epsilon^{-1})$, $B_0 = \mathcal{O}(\epsilon^{-0.5}), B_1= \mathcal{O}(\epsilon^{-0.5})$, and $N = \mathcal{O}(\epsilon^{-1})$, We can ensure that $\E\left[\Norm{\G(\x_t,\beta)}^2\right]\leq \epsilon$. This guarantee can also be satisfied by setting $\alpha =\mathcal{O}\left( \epsilon\right), \eta = \mathcal{O}\left( \sqrt{\epsilon}\right),T = \mathcal{O}\left( \epsilon^{-1.5} \right), B_0 = \mathcal{O}(\epsilon^{-0.5}), B_1 =  \mathcal{O}\left(1 \right), N = \mathcal{O}(\epsilon^{-1})$.

\section{Proof of Optimal Gap~(Theorem~\ref{thm:3},~\ref{thm:3_0},~\ref{thm:4} and \ref{thm:4_0})}
In this section, we investigate the optimal gap for convex and strongly convex objective functions.
\subsection{Convex}
According to the equation~(C.21) of \citet{pmlr-v97-yurtsever19b}, for algorithm 1 with convex objectives, we have that
\begin{align*}
\mathbb{E}\left[F\left(\x_{t+1}\right)\right]-F_{\star} \leq\left(1-\eta\right)\left(\mathbb{E}\left[F\left(\x_{t}\right)\right]-F_{\star}\right)+\eta D \mathbb{E}\left\|\nabla F\left(\x_{t}\right)-\v_{t}\right\|+\eta^{2}\frac{L_F D^2}{2}  
\end{align*}
Then we have:
\begin{align*}
 &\frac{1}{T}\sum_{t=1}^T \mathbb{E}\left[F\left(\x_{t+1}\right)\right]-F_{\star} \\
 \leq &\frac{\left(\mathbb{E}\left[F\left(\x_{1}\right)\right]-F_{\star}\right)}{\eta T}+\frac{ D}{T} \sum_{t=1}^T \mathbb{E}\left\|\nabla F\left(\x_{t}\right)-\v_{t}\right\|+\eta \frac{L_F D^2}{2} \\
  \leq & \frac{\left(\mathbb{E}\left[F\left(\x_{1}\right)\right]-F_{\star}\right)}{\eta T}+D\sqrt{\frac{ 1}{T} \sum_{t=1}^T \mathbb{E}\left\|\nabla F\left(\x_{t}\right)-\v_{t}\right\|^2}+\eta \frac{L_F D^2}{2} \\
  \leq & \frac{\left(\mathbb{E}\left[F\left(\x_{1}\right)\right]-F_{\star}\right)}{\eta T}+D\sqrt{\frac{\E\left[L_2\Norm{\v_1 -\prod_{i=1}^K \nabla f_i(\u_1^{i-1})}^2\right] + 2L_2\sum_{i=1}^K \E\left[\Norm{\u_1^i -  f_i(\u_1^{i-1})}^2\right] }{\alpha T}}\\
  & +D\sqrt{\frac{\alpha L_3}{B_1} + \frac{\eta^2  L_3}{\alpha B_1}} +\eta \frac{L_F D^2}{2} 
\end{align*}
Next, we denote that $\Gamma_s = \frac{1}{T}\sum_{t=1}^{T} \Norm{\v_t -\prod_{i=1}^K \nabla f_i(\u_t^{i-1})}^2 + \frac{2}{T}\sum_{t=1}^{T}\sum_{i=1}^K \Norm{\u_t^i -  f_i(\u_t^{i-1})}^2 $ for stage $s$ and we denote $\x^s$ as the output of Algorithm 3 for the stage $s$. Then, we have:
\begin{align*}
\mathbb{E}\left[F\left(\x^{s}\right)-F_{\star}\right] \leq \frac{\mathbb{E}\left[F\left(\x^{s-1}\right)-F_{\star}\right]}{\eta_s T_s}+D\sqrt{\frac{L_2 \Gamma_{s-1} }{\alpha_s  T_s}} +D\sqrt{\frac{\alpha_s L_3}{B_1^s} + \frac{\eta_s^2  L_3}{\alpha_s B_1^s}} +\eta_s \frac{L_F D^2}{2} 
\end{align*}
Also, we have that:
\begin{align*}
    \Gamma_{s} \leq&  \frac{2\Gamma_{s-1}}{\alpha_{s} T_{s}} + \frac{\alpha_{s} L_3}{B_1^{s}} + \frac{L_3 \eta_{s}^2}{\alpha_{s} B_1^{s}}
\end{align*}
Set $\epsilon_s = \left(\frac{1}{2}\right)^{s-1} $, $\eta_s = \alpha_s \leq \frac{\epsilon_s}{4L_F D^2}, T_s \geq \max\left\{\frac{(4\Delta_F+24+64D^2L_2)}{\eta_s}, \frac{\Gamma_0 (16D^2L_2+6)}{\eta_s}\right\},B_1^s \geq \max \left\{\frac{12\eta_s L_3}{\epsilon_s^2},\frac{128 L_3 \eta_s D^2}{\epsilon_s^2} \right\}$. We can guarantee that $\mathbb{E}\left[F\left(\x^{s}\right)-F_{\star}\right] \leq \epsilon_s$ and $\E\left[\Gamma_{s} \right] \leq \epsilon_s^2$.
We will use the  induction to give the proof:
\begin{proof}
When $s=1$, we have:
\begin{align*}
 \mathbb{E}\left[F\left(\x^{1}\right)-F_{\star}\right] & \leq \frac{\Delta_F}{\eta_1 T_1}+D\sqrt{\frac{L_2 \Gamma_{0} }{\alpha_1  T_1}} +D\sqrt{\frac{\alpha_1 L_3}{B_1^1} + \frac{\eta_1^2  L_3}{\alpha_1 B_1^1}} +\eta_1 \frac{L_F D^2}{2} \\
 & \leq \epsilon_1 =1 
\end{align*}
where $\Gamma_0 = K^2 \sigma_J^2 L_f^{2K-2} + 2K \sigma^2$. Also, we have that:
\begin{align*}
     \Gamma_1 &\leq \frac{2\Gamma_0}{\alpha_1 T_1} + \frac{\alpha_1 L_3}{B_1^1} + \frac{L_3 \eta_1^2}{\alpha_1 B_1^1}\leq  \epsilon_1^2 = 1 
 \end{align*}

Then, assume  $\mathbb{E}\left[F\left(\x^{s}\right)-F_{\star}\right] \leq \epsilon_s$ and $\mathbb{E}\left[\Gamma_s\right] \leq \epsilon_s^2$, we would prove that it holds for stage $s+1$.
\begin{align*}
& \mathbb{E}\left[F\left(\x^{s+1}\right)-F_{\star}\right]  \\
&\leq \frac{\epsilon_s}{\eta_{s+1} T_{s+1}}+D\sqrt{\frac{L_2 \epsilon_s^2 }{\alpha_{s+1}  T_{s+1}}} +D\sqrt{\frac{\alpha_{s+1} L_3}{B_1^{s+1}} + \frac{\eta_{s+1}^2  L_3}{\alpha_{s+1} B_1^{s+1}}} +\eta_{s+1} \frac{L_F D^2}{2} \\
&\leq \frac{\epsilon_s}{8}+\frac{\epsilon_s}{8}+\frac{\epsilon_s}{8}+\frac{\epsilon_s}{8} = \epsilon_{s+1}
\end{align*}
We also know that
\begin{align*}
    \Gamma_{s+1} \leq&  \frac{2\epsilon_s^2}{\alpha_{s+1} T_{s+1}} + \frac{\alpha_{s+1} L_3}{B_1^{s+1}} + \frac{L_3 \eta_{s+1}^2}{\alpha_{s+1} B_1^{s+1}}\\
    \leq & \frac{\epsilon_s^2}{12}+\frac{\epsilon_s^2}{12}+\frac{\epsilon_s^2}{12}
    =  \epsilon_{s+1}^2
\end{align*}

So we prove $\mathbb{E}\left[F\left(\x^{s}\right)-F_{\star}\right] \leq \left(\frac{1}{2}\right)^{s-1}$ with $\eta_s = \alpha_s \leq \frac{\epsilon_s}{4L_F D^2},B_1^s \geq \max \left\{\frac{12\eta_s L_3}{\epsilon_s^2},\frac{128 L_3 \eta_s D^2}{\epsilon_s^2} \right\}$, $T_s \geq \max\left\{\frac{(4\Delta_F+24+64D^2L_2)}{\eta_s}, \frac{\Gamma_0 (16D^2L_2+6)}{\eta_s}\right\}$. This condition can be satisfied by setting that $\eta_s = \mathcal{O} (\epsilon_s)$, $T_s = \mathcal{O}(\epsilon_s^{-1})$, $\alpha = \mathcal{O} (\epsilon_s)$ and $B_1^s = \mathcal{O} (\epsilon_s^{-1})$ [Large batch version]. This can also be achieved by setting that $\eta_s = \mathcal{O} (\epsilon_s^2)$, $T_s = \mathcal{O}(\epsilon_s^{-2})$, $\alpha = \mathcal{O} (\epsilon_s^2)$ and $B_1^s = \mathcal{O} (1)$ [Constant Batch].

To ensure $\mathbb{E}\left[F\left(\x^{s}\right)-F_{\star}\right] \leq \epsilon$, set $S=\mathcal{O}(\log_2 (\frac{1}{\epsilon}))$, and the SFO rate is $\sum_{s=1}^S T^s B_1^s = \mathcal{O}\left(\sum_{s=1}^S 2^{(2s)}\right) = \mathcal{O}\left(\frac{1}{\epsilon^2}\right)$. 
\end{proof}

\subsection{Strongly Convex}
In this section, we assume $F(\x)$ is $\lambda$-strongly convex function and we set $\eta \leq \frac{\lambda}{4L_F}$:

\begin{align*} 
&F\left(\x_{t+1}\right)\\ 
\leq & F\left(\x_t\right)+\left\langle\nabla F\left(\x_t\right), \x_{t+1}-\x_t\right\rangle+\frac{L_F}{2}\left\|\x_{t+1}-\x_{t}\right\|^{2} \\ 
\leq &F\left(\x_{t}\right)+\eta\left\langle\nabla F\left(\x_{t}\right), \z_{t}-\x_{t}\right\rangle+\eta^{2} \frac{L_F}{2} \Norm{\z_t - \x_t}^{2} \\ 
=&F\left(\x_{t}\right)+\eta\left\langle \v_{t}, \z_{t}-\x_{t}\right\rangle+\eta\left\langle\nabla F\left(\x_{t}\right)-\v_t, \z_t-\x_t\right\rangle + \frac{\eta \lambda}{4} \Norm{\z_t - \x_t}^{2} +(\eta^{2} \frac{L_F}{2} - \frac{\eta \lambda}{4})\Norm{\z_t - \x_t}^{2} \\ 
\leq & F(\x_t) + \eta \left\langle \nabla F(\x_t) - \v_t, \z_t - \x^{\star} \right\rangle + \eta \left\langle \nabla F(\x_t) , \x^{\star} - \x_t \right\rangle +  \frac{\eta\lambda}{4} \Norm{\x^{\star}-\x_t}^2 + \frac{\eta \lambda D^2}{N} - \frac{\eta \lambda}{8} \Norm{\z_t - \x_t}^2,
\end{align*}    
where the last inequality is due to the fact that
\begin{align*}
    \eta\left\langle \v_{t}, \z_{t}-\x_{t}\right\rangle + \frac{\eta \lambda}{4} \Norm{\z_t - \x_t}^{2}  \leq \eta\left\langle \v_{t}, \x^{\star}-\x_{t}\right\rangle + \frac{\eta \lambda}{4} \Norm{\x^{\star} - \x_t}^{2}  + \frac{\eta \lambda D^2}{N}
\end{align*}
For $\lambda$-strongly convex function, we have $\left\langle \nabla F(\x_t), \x^{\star} - \x_t \right\rangle \leq F_{\star} - F(\x_t) - \frac{\lambda}{2} \Norm{\x_t - \x^{\star}}^2$. As a result:

\begin{align*} 
F\left(\x_{t+1}\right) 
& \leq F(\x_t) + \eta (F_{\star} - F(\x_t)) - \frac{\eta \lambda}{4}\Norm{\x^{\star} - \x_t}^2 + \frac{\eta \lambda}{16} \Norm{\z_t - \x^{\star}}^2 + \frac{4\eta}{\lambda} \Norm{\nabla F(\x_t) - \v_t}^2\\
& \qquad + \frac{\eta \lambda D^2}{N} - \frac{\eta \lambda}{8} \Norm{\z_t - \x_t}^2
\end{align*}    

So we have:
$$F\left(\x_{t+1}\right)  - F_{\star} \leq (1-\eta) (F\left(\x_{t}\right) -F_{\star} ) +\frac{4\eta}{\lambda} \Norm{\nabla F(\x_t) - \v_t}^2 + \frac{\eta \lambda D^2}{N}$$

Finally, we have: 
\begin{align*}
    \frac{1}{T}\sum_{i=1}^T (F\left(\x_{t}\right) -F_{\star} ) & \leq \frac{F\left(\x_1\right) - F_{\star}}{\eta T} + \frac{4}{\lambda T}\sum_{t=1}^T \left\|\nabla F\left(\x_t\right)-\v_t\right\|^2+\frac{ \lambda D^2}{N}\\
\end{align*}
Next, we denote $\x^s$ as the output for stage $s$. Then, we have:
\begin{align*}
\mathbb{E}\left[F\left(\x^{s}\right)\right]-F_{\star} \leq \frac{F\left(\x_{s-1}\right) - F_{\star}}{\eta_s T_s} + \frac{4\Gamma_{s-1}}{\lambda \alpha_s T_s } + \frac{4 \alpha_s L_3}{\lambda B_1^s} + \frac{4 \eta^2_s  L_3}{\lambda \alpha_s B_1^s} +  \frac{\lambda D^2}{N}
\end{align*}
Also, we have that:
\begin{align*}
    \Gamma_{s} \leq&  \frac{2\Gamma_{s-1}}{\alpha_{s} T_{s}} + \frac{\alpha_{s} L_3}{B_1^{s}} + \frac{L_3 \eta_{s}^2}{\alpha_{s} B_1^{s}}
\end{align*}

Set that $\epsilon_s  = (\frac{1}{2})^{s-1},    B_1^s \geq  \frac{ 72 \alpha_s L_3 }{ \lambda \epsilon_s }, N \geq \frac{6\lambda D^2}{\epsilon_s}, T_s \geq \max \{ \frac{72(\Delta_F+1)}{\eta_s}, \frac{ 72 (D^2 +1)\Gamma_0}{\eta_s} \}$, $B_0 =\max \{ \lambda^{-1},1\}$. We can guarantee that $\mathbb{E}\left[F\left(\x^{s}\right)-F_{\star}\right] \leq \epsilon_s$ and $\E\left[\Gamma_{s} \right] \leq \lambda \epsilon_s$.
We will use the  induction to give the proof:

\begin{proof}
When $s=1$, we have: 
\begin{align*}
\Gamma_{1}  \leq & \frac{2\Gamma_{0}}{\alpha_{1} T_{1} B_0} + \frac{\alpha_{1} L_3}{B_1^{1}} + \frac{L_3 \eta_{1}^2}{\alpha_{1} B_1^{1}}
\leq   \lambda \epsilon_1 = \lambda ,
\end{align*}
where $B_0$ denotes the batch size used only in the first iteration of the first stage. Also, we have that:
\begin{align*}
 \mathbb{E}\left[F\left(\x^{1}\right)-F_{\star}\right] & \leq \frac{\Delta_F}{\eta_1 T_1}+\frac{4\Gamma_0}{\lambda \alpha_1 T_1 B_0} + \frac{4 \alpha_1 L_3}{\lambda B_1^1} + \frac{4 \eta^2_1  L_3}{\lambda \alpha_1 B_1^1} +  \frac{\lambda D^2}{N} \\
 & \leq \epsilon_1 =1 
\end{align*}

Then assume $\mathbb{E}\left[\Gamma_s\right] \leq \lambda \epsilon_{s}$ and $\mathbb{E}\left[F\left(\x^{s}\right)-F_{\star}\right] \leq \epsilon_s$, we would prove that it holds for stage $s+1$.
\begin{align*}
    \Gamma_{s+1} \leq&  \frac{2\Gamma_{s}}{\alpha_{s+1} T_{s+1}} + \frac{\alpha_{s+1} L_3}{B_1^{s+1}} + \frac{L_3 \eta_{s+1}^2}{\alpha_{s+1} B_1^{s+1}}\\
    \leq & \frac{2\lambda \epsilon_s}{\alpha_{s+1} T_{s+1}} + \frac{\alpha_{s+1} L_3}{B_1^{s+1}} + \frac{L_3 \eta_{s+1}^2}{\alpha_{s+1} B_1^{s+1}}\\
    \leq & \frac{\lambda \epsilon_s}{2} = \lambda \epsilon_{s+1}
\end{align*}
Besides, we know that
\begin{align*}
\mathbb{E}\left[F\left(\x^{s+1}\right)-F_{\star} \right]
& \leq \frac{\epsilon_s}{\eta_{s+1} T_{s}} + \frac{4\Gamma_{s}}{\lambda \alpha_{s+1} T_{s+1}} +  \frac{4 \alpha_{s+1} L_3}{\lambda  B_1^{s+1}} + \frac{4 \eta^2_{s+1}  L_3}{\lambda \alpha_{s+1} B_1^{s+1}} +  \frac{\lambda D^2}{N}\\
& \leq  \frac{\epsilon_{s}}{2} = \epsilon_{s+1}
\end{align*}
\end{proof}

So we prove $\mathbb{E}\left[F\left(\x^{s}\right)-F_{\star}\right] \leq \left(\frac{1}{2}\right)^{s-1}$ with $\epsilon_s  = (\frac{1}{2})^{s-1},    B_1^s \geq  \frac{ 72 \alpha_s L_3 }{ \lambda \epsilon_s }, N \geq \frac{6\lambda D^2}{\epsilon_s}, T_s \geq \frac{72(\Delta_F+1+\Gamma_0)}{\eta_s}$, $B_0 =\max \{ \lambda^{-1},1\}$.

This condition can be satisfied by setting that $\eta_s= \alpha_s = \mathcal{O} (\lambda)$, $T_s = \mathcal{O}(\lambda^{-1})$, $B_1 = \mathcal{O} (\epsilon_s^{-1}), N=\mathcal{O}(\frac{\lambda}{\epsilon_s})$  (Large Batch) or by setting that $\eta_s = \alpha_s= \mathcal{O} (\lambda \epsilon_s)$, $T_s = \mathcal{O}(\lambda^{-1}\epsilon_s^{-1})$, $B_1 = \mathcal{O} (1), N=\mathcal{O}(\frac{\lambda}{\epsilon_s})$  (Constant Batch). To ensure $\mathbb{E}\left[F\left(\x^{s}\right)-F_{\star}\right] \leq \epsilon$, set $S=\mathcal{O}(\log_2 (\frac{1}{\epsilon}))$, and the SFO rate is $\sum_{s=1}^S T_s B_1^s = \sum_{s=1}^S \frac{\mathcal{O}(1)}{\lambda} 2^{s-1} =\mathcal{O}(\frac{1}{\lambda \epsilon})$.

\end{document}